%% file: main.tex
\documentclass[journal,twoside,web]{ieeecolor}

\usepackage{mathdots}

\let\labelindent\relax
\usepackage[tbtags]{amsmath}
\usepackage{enumitem}
\usepackage{generic}
\usepackage{comment}
\usepackage{mathtools}
\mathtoolsset{showonlyrefs}
\usepackage{tabularx} 
\usepackage{diagbox}
\usepackage{graphicx}
\usepackage{mathrsfs}
\usepackage[noadjust]{cite}
\usepackage{amssymb,amsfonts}
\usepackage{algorithmic}
\usepackage{tikz}
\usetikzlibrary{shapes,arrows,patterns,positioning}
\tikzstyle{block} = [draw, fill=gray!20, rectangle, 
    minimum height=2em, minimum width=4em]
\tikzstyle{sum} = [draw, fill=gray!20, circle, node distance=1.5cm]
\tikzstyle{input} = [coordinate]
\tikzstyle{output} = [coordinate]
\tikzstyle{pinstyle} = [pin edge={to-,thin,black}]

\DeclareMathOperator{\rank}{rank}
\DeclareMathOperator{\tr}{tr}

\DeclareMathOperator{\rad}{rad}
\DeclareMathOperator{\cent}{cent}

\DeclareMathOperator{\inter}{int}

\DeclareMathOperator{\spec}{spec}

\newcommand{\set}[2]{\left\{#1 \mid #2\right\}}
\newcommand{\norm}[1]{\left\|#1\right\|}

\usepackage{float}
\usepackage{accents}
\usepackage{multicol}
\usepackage{multirow}
\usepackage[ruled,vlined]{algorithm2e}

\usepackage{array,multirow}
\usepackage{hhline}
\usepackage{arydshln}

\newtheorem{theorem}{Theorem}
\newtheorem{corollary}[theorem]{Corollary}
\newtheorem{lemma}[theorem]{Lemma}
\newtheorem{example}[theorem]{Example}

\newtheorem{proposition}[theorem]{Proposition}

\newtheorem{remark}[theorem]{Remark}

\newtheorem{definition}[theorem]{Definition}

\usepackage{textcomp}
\def\BibTeX{{\rm B\kern-.05em{\sc i\kern-.025em b}\kern-.08em
    T\kern-.1667em\lower.7ex\hbox{E}\kern-.125emX}}
\begin{document}
\title{Experiment Design for Set-membership Identification: From Prior Knowledge to Universal Inputs}
\author{Amir Shakouri, 
Henk J. van Waarde, M. Kanat Camlibel
\thanks{The work of Henk van Waarde was supported by
the Dutch Research Council under the NWO Talent Programme Veni
Agreement (VI.Veni.22.335).}
\thanks{The authors are with the Bernoulli Institute for Mathematics, Computer Science and Artificial Intelligence, University of Groningen (e-mail: a.shakouri@rug.nl, h.j.van.waarde@rug.nl, m.k.camlibel@rug.nl). }
}

\maketitle

\thispagestyle{empty}
\pagestyle{empty}

\begin{abstract}
We consider the problem of designing input signals for an unknown linear time-invariant system in such a way that the resulting data, within a finite horizon, is suitable for identification with a desired accuracy. We consider both noise-free and noisy settings with $\ell_\infty$--bounded noise models. We will take into account general prior knowledge of the system parameters. Central in our study is the concept of \emph{universal inputs}. An input is called universal for identification if, when applied to any system complying with the prior knowledge, it yields data suitable for accurate identification. We provide new methods for designing such universal inputs. Our results generalize the experiment design approach based on Willems et al.'s fundamental lemma that relies on persistently exciting inputs, and that is limited to prior knowledge on controllability. It turns out that for other types of prior knowledge, there exist universal inputs that outperform the persistently exciting ones, e.g., in terms of sample efficiency. Moreover, we investigate types of prior knowledge that enable experiment design for \emph{exact} identification in the presence of noise.
\end{abstract}


\input{introduction}

\input{preliminaries}




\input{UED}

\input{UED_noisy}

\input{i-o}



\input{examples}

\input{appendix}

\section*{References}

\bibliographystyle{IEEEtran}
\bibliography{biblo}

\begin{IEEEbiography}[{\includegraphics[width=1in,height=1.25in,clip,keepaspectratio]{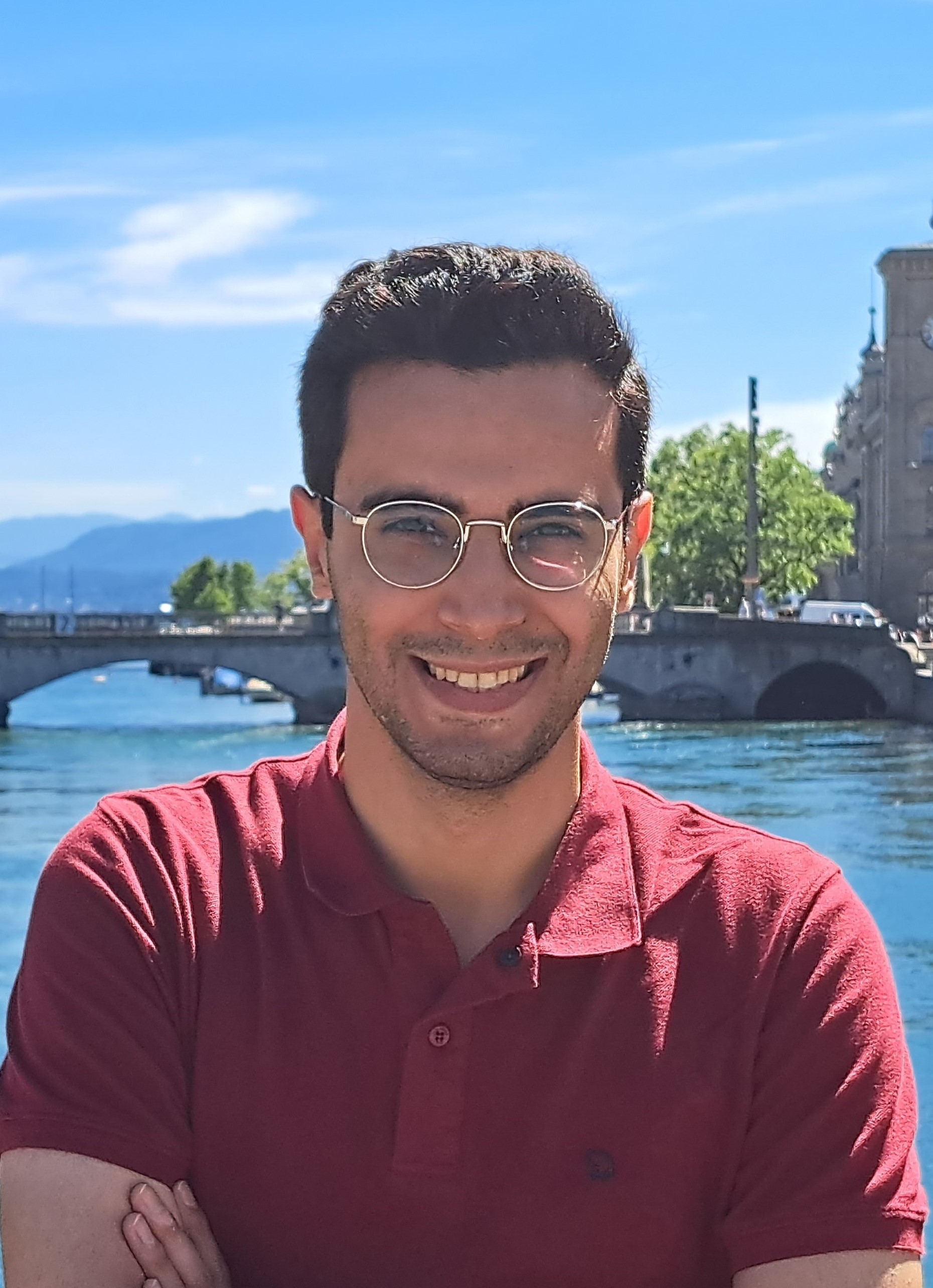}}]{Amir Shakouri} is currently pursuing the Ph.D. degree in applied mathematics with the Systems, Control, and Optimization Group, Bernoulli Institute for Mathematics, Computer Science and Artificial Intelligence, University of Groningen, Groningen, The Netherlands. His research interests include robust control, data-driven control, and experiment design. He is a recipient of the 2025 CDC Outstanding Student Paper Award.
\end{IEEEbiography}

\begin{IEEEbiography}[{\includegraphics[width=1in,height=1.25in,clip,keepaspectratio]{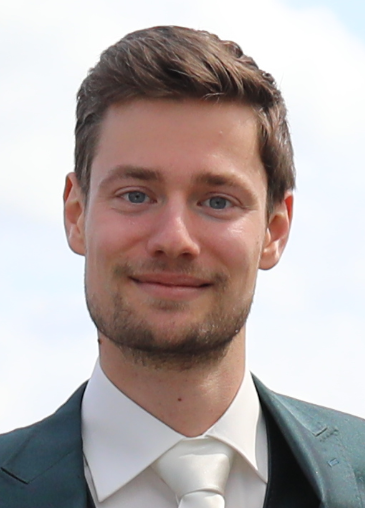}}]{Henk J. van Waarde} is an assistant professor in the Bernoulli Institute for Mathematics, Computer Science and Artificial Intelligence at the University of Groningen, The Netherlands. During 2020-2021 he was a postdoctoral researcher, first at the University of Cambridge, UK, and later at ETH Zurich, Switzerland. He obtained the Ph.D. degree cum laude in Applied Mathematics from the University of Groningen in 2020. He was also a visiting researcher at the University of Washington, Seattle in 2019-2020. 

His research interests include learning and data-driven control, system identification and identifiability, networks of dynamical systems, and robust and optimal control. Dr. van Waarde is the recipient of the 2025 SIAM Activity Group on Control and Systems Theory Prize. He serves as an Associate Editor of the IEEE Control Systems Letters.
\end{IEEEbiography}

\begin{IEEEbiography}[{\includegraphics[width=1in,height=1.25in,clip,keepaspectratio]{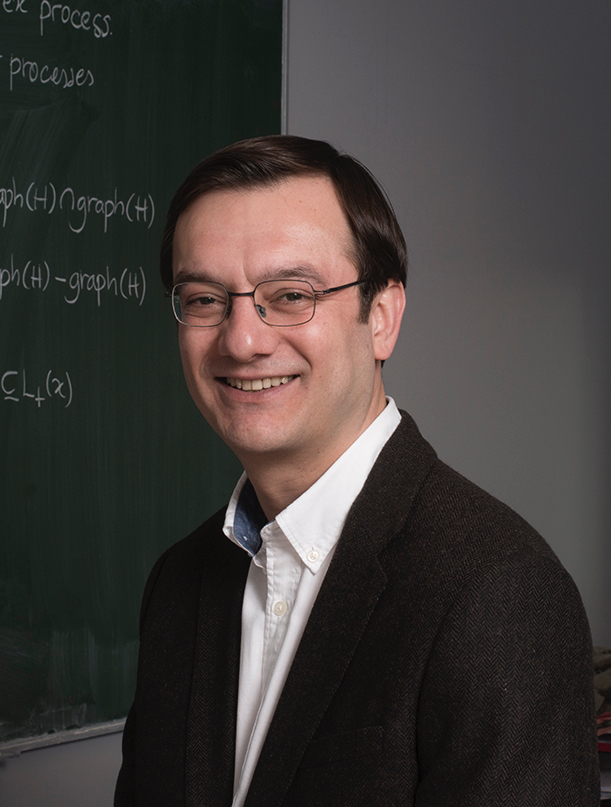}}]{M. Kanat Camlibel} is a Professor of Applied Mathematics at the University of Groningen, where he chairs the research group Systems, Control, and Optimization. He obtained his Ph.D. degree from Tilburg University (2001) and both Bachelor's and Master's degrees in Control and Computer Engineering from Istanbul Technical University (1991 and 1994). Prior to his current position, he served as an Assistant Professor at Eindhoven University of Technology and held postdoctoral positions at the University of Groningen and Tilburg University.

He has served as an Associate Editor for the International Journal of Robust and Nonlinear Control, Systems \& Control Letters, SIAM Journal on Control and Optimization, and IEEE Transactions on Automatic Control. His research interests encompass a wide range of areas, including complementarity systems, piecewise affine dynamical systems, switched linear systems, constrained linear systems, multi-agent systems, hybrid systems, dynamical networks, model reduction, geometric theory of linear systems, and data-driven control.
\end{IEEEbiography}

\end{document}

%% file: introduction.tex
\section{Introduction}
\label{sec:introduction}

\IEEEPARstart{S}{ystem} identification deals with the problem of constructing mathematical models of dynamical systems using measured input-output data and prior knowledge of the system (see, e.g., \cite{goodwin1977dynamic,ljung1998system,van2012subspace,milanese2013bounding,pintelon2012system,DBLSCT2025}). Obtaining an accurate model using system identification requires a sufficiently informative dataset. Therefore, an important problem is that of \emph{experiment design}, which deals with the question of how to select the input signal in such a way that the resulting data is informative. Answers to this question rely on the given prior knowledge of the system \cite[Ch. 6]{goodwin1977dynamic}. Two main approaches are commonly considered in experiment design: \emph{online} methods, in which the input at each time step is chosen based on data collected in the previous steps \cite{borchers2011set,karimshoushtari2020design,van2021beyond,camlibel2025shortest} (also see \cite[Sec. 11.3]{hjalmarsson2005experiment} and the references therein); \emph{offline} methods, in which the entire input signal is designed \emph{a priori} and applied without modification during the experiment \cite{bai1985persistency,green1986persistence,willems2005note,coulson2022quantitative,berberich2023quantitative,rojas2007robust,gevers1986optimal,deistler1995consistency,weyer2002finite,campi2002finite,venkatasubramanian2025beyond,pronzato1989experiment,borchers2011design,dahleh1993sample,poolla1994time}.

In this work, we propose a new framework for studying the offline experiment design problem using general prior knowledge of the system. We consider both the noise-free scenario, where we aim at exact system identification, and the noisy setting, where we aim at system identification up to a prescribed accuracy. The central concept in our study is the notion of \emph{universal inputs}. An input is called universal if it leads to informative data when applied to \emph{any} system compatible with the prior knowledge. In this paper, we provide novel conditions under which an input is universal across a wide range of different types of prior knowledge. As we will demonstrate, there are many instances of prior knowledge for which our methods outperform existing approaches based on persistency of excitation \cite{bai1985persistency,green1986persistence,willems2005note,coulson2022quantitative,berberich2023quantitative}, e.g., in terms of sample efficiency. This is particularly important for cases where applying persistently exciting inputs is not practical due to safety considerations, physical constraints, or limited resources, as encountered in chemical/biological systems \cite{sontag2003differential}, clinical settings \cite{engelhart2016comparison}, large-scale networks \cite{tsiamis2023statistical}, and aerospace systems \cite{prussing1970optimal}.

\subsection{Related work}

Linear system identification has been studied within different approaches depending on the system class and the noise model. Prediction error methods \cite{ljung1998system}, subspace identification \cite{van2012subspace}, identification in the frequency domain \cite{pintelon2012system}, and set-membership identification \cite{milanese2013bounding} are among the most popular approaches. Of particular relevance to this paper is set-membership identification, where the focus is on bounded deterministic noise models, and where the aim is to obtain a model with the smallest \emph{worst-case error} (see, e.g., \cite{milanese1991optimal,li2024learning,shakouri2025chebyshev,lauricella2020set} and the references therein). 

Several types of prior knowledge have been incorporated into system identification. For instance, bounds on the system parameters have been considered, e.g., in \cite{rojas2007robust} (see also \cite[Sec. 2]{hjalmarsson2005experiment}). System-theoretic properties such as minimality \cite{verhaegen1992subspace}, stability \cite{van2002identification,lacy2003subspace,miller2013subspace}, positivity \cite{de2002identification}, and passivity \cite{goethals2003identification,rodrigues2021novel} are also among the investigated types of prior knowledge. Moreover, incorporating information on the system's transfer function has been studied in \cite{inoue2019subspace,khosravi2023kernel}. 

Experiment design has also been investigated from various perspectives. For stochastic noise models, asymptotic methods have been studied, e.g., in \cite{rojas2007robust,gevers1986optimal,deistler1995consistency}, where the goal is to find infinite-horizon input signals such that the identification error decreases as the number of samples tends to infinity. Nonasymptotic methods, on the other hand, aim to guarantee a desired identification accuracy using \emph{finite} data. In the stochastic setting, such methods have been studied, e.g., in \cite{weyer2002finite,campi2002finite,venkatasubramanian2025beyond}. 

Unlike the previously mentioned works \cite{rojas2007robust,gevers1986optimal,deistler1995consistency,weyer2002finite,campi2002finite,venkatasubramanian2025beyond}, in this paper, we focus on bounded deterministic noise models. In this setting, asymptotic approaches using online \cite{karimshoushtari2020design,borchers2011set} and offline \cite{pronzato1989experiment,borchers2011design} methods are among the investigated topics. Literature on nonasymptotic experiment design for set-membership identification with provable guarantees on the identification accuracy and sample complexity is so far limited only to finite impulse response (FIR) systems \cite{dahleh1993sample,poolla1994time}. This problem is deemed to be challenging in general and has not been thoroughly studied in the literature. It was not until recently that important steps towards solving this problem have been taken in \cite{coulson2022quantitative} (also see \cite{berberich2023quantitative} and \cite[Sec. 4.5]{shakouri2025chebyshev}), where a robust version of Willems et al.'s fundamental lemma is developed. However, the framework in \cite{coulson2022quantitative} is confined to controllability and does not accommodate other forms of prior knowledge.

\subsection{Contributions}

In this paper, we study offline experiment design for set-membership identification with guarantees on the identification accuracy and the number of data samples. We investigate the role of prior knowledge in this problem. We first consider the noise-free setting, and then we extend our results to the noisy case with $\ell_\infty$--bounded noise models. Our contributions are summarized as follows:
\begin{enumerate}[wide]
    \item A framework is developed for studying the experiment design problem, leveraging the concepts of data informativity and input universality by taking prior knowledge into account.
    \item We show that every universal input must be persistently exciting of a sufficiently high order in the following cases: single-input systems with open sets of prior knowledge (Theorem~\ref{th:single-input-noiseless}); and multi-input systems with the prior knowledge in the form of a norm bound on the input matrix (Theorem~\ref{th:PE_nec_ball}). 
    \item We study cases where the prior knowledge allows for designing universal inputs that are not necessarily persistently exciting (see Theorems~\ref{th:GM_noiseless} and~\ref{th:GM_noiseless_generalized}). For instance, our results enable universal experiment design using hands-off input signals that take zero value over a significant interval of the experiment horizon (see Section~\ref{subsec:ex1}). 
    \item In the noisy setting, we study experiment design for identification with a desired accuracy for which several design methods are provided (see Theorems~\ref{th:open_design_noisy},~\ref{th:open_enabling_noisy}, and~\ref{th:helping_in_exp}). For instance, our results enable experiment design with fewer data samples compared to available methods based on persistency of excitation in \cite{willems2005note,coulson2022quantitative} (see Section~\ref{sub:network}).
    \item We further study experiment design for \emph{exact} identification in the presence of noise. We show that this problem is feasible only if the prior knowledge set is \emph{uniformly discrete} (see Theorems~\ref{th:GM_noisy_generalized}). For such sets of prior knowledge, we present an experiment design method for exact identification (see Theorem~\ref{th:ctrl_nec_noisy_unique}). 
\end{enumerate}

\subsection{Paper organization}

This paper is organized as follows: Section~\ref{sec:pre} provides the preliminaries. In Section~\ref{sec:prob}, we formulate the problem of universal experiment design. In Section~\ref{sec:UED_noisefree}, we study the experiment design using input-state data in the noise-free setting. We investigate the noisy setting in Section~\ref{sec:i-o gain}. We discuss universal experiment design based on input-output data in Section~\ref{sec:i-o data}. Two benchmark examples are considered in Section~\ref{sec:ex}, and finally, Section~\ref{sec:conclusions} concludes the paper.


%% file: preliminaries.tex
\section{Preliminaries}
\label{sec:pre}

\subsection{Notation}

An integer interval between $a\in\mathbb{Z}$ and $b\in\mathbb{Z}$ with $b\geq a$ is denoted by $[a,b]\coloneqq\set{x\in\mathbb{Z}}{a\leq x\leq b}$. We denote the set of nonnegative integers by $\mathbb{Z}_+$. The $p$-norm of a vector $v\in\mathbb{R}^n$ is denoted by $\norm{v}_p$, and in particular, the Euclidean $2$-norm is denoted by $\norm{v}$. 

For a matrix \mbox{$M\in\mathbb{R}^{n\times m}$}, we denote its singular values by \mbox{$\sigma_1(M)\geq\cdots\geq \sigma_{\min\{m,n\}}(M)$}. Moreover, $\sigma_i(M)=0$ if \mbox{$i>\min\{m,n\}$}. Let $\sigma_*(M)$ denote the smallest positive singular value of $M$ if $M\neq 0$, and $\sigma_*(M)=0$ if $M=0$. The spectral norm of a matrix $M\in\mathbb{R}^{n\times m}$ is denoted by $\norm{M}=\sigma_1(M)$. If $M\in\mathbb{R}^{n\times n}$ is symmetric and all its eigenvalues are positive (resp., nonnegative), we say that $M$ is positive definite (resp., positive semi-definite) and we denote it by $M>0$ (resp., $M\geq 0$). We say $M$ is negative definite (resp., negative semi-definite) and we denote it by $M<0$ (resp., $M\leq 0$) if \mbox{$-M>0$} (resp., $-M\geq 0$). We denote the Moore–Penrose pseudoinverse of $M\in\mathbb{R}^{n\times m}$ by $M^\dagger$. The Kronecker product of two matrices $M$ and $N$ is denoted by $M\otimes N$. 

Given $T\in\mathbb{Z}_+$ and $v:[0,T-1]\rightarrow\mathbb{R}^q$, we define
\begin{equation}
v_{[0,T-1]}\coloneqq\begin{bmatrix}
v(0)^\top & v(1)^\top & \cdots & v(T-1)^\top
\end{bmatrix}^\top.
\end{equation}
For $v_{[0,T-1]}\in\mathbb{R}^{qT}$, its \emph{Hankel matrix} of depth $k\in[1,T]$ is denoted by
\begin{equation}
\mathcal{H}_k(v_{[0,T-1]})\coloneqq \begin{bmatrix}
v_{[0,k-1]} & v_{[1,k]} & \cdots & v_{[T-k,T-1]}
\end{bmatrix}.
\end{equation}
We say that $v_{[0,T-1]}$ is \emph{persistently exciting of order $k$} if $\mathcal{H}_k(v_{[0,T-1]})$ has full row rank.

\subsection{Chebyshev centers and radii}

Let $\mathcal{X}\subset\mathbb{R}^{p\times q}$ be nonempty and bounded. We define 
\begin{equation}
\label{eq:cheb}
\rad \mathcal{X}\coloneqq \inf_{C\in\mathcal{X}} \sup_{X\in\mathcal{X}} \norm{C-X}
\end{equation}
as the \emph{Chebyshev radius} of $\mathcal{X}$ and
\begin{equation}
\label{eq:cent}
\cent\mathcal{X}\coloneqq\set{C\in\mathcal{X}}{\norm{C-X}\leq\rad\mathcal{X} \textup{ for all }X\in\mathcal{X}}
\end{equation}
as the \emph{set of Chebyshev centers} of $\mathcal{X}$. The Chebyshev radius is the smallest radius of a ball containing a set. The center of such a ball is called a Chebyshev center. Due to boundedness of $\mathcal{X}$, $\rad\mathcal{X}$ is well-defined and $\cent\mathcal{X}$ is nonempty. 

For a bounded set of matrix pairs $\Sigma_s\subset\mathbb{R}^{n\times n}\times \mathbb{R}^{n\times m}$, let $\mathcal{S}\coloneqq\set{\begin{bmatrix}
A \!\!&\!\! B
\end{bmatrix}\in\mathbb{R}^{n\times(n+m)}}{(A,B)\in\Sigma_s}$. We define $\rad\Sigma_s\coloneqq\rad\mathcal{S}$ and
\begin{equation}
\cent\Sigma_s\coloneqq\set{(A,B)\in\Sigma_s}{\begin{bmatrix}
A \!\!&\!\! B
\end{bmatrix}\in\cent\mathcal{S}}.
\end{equation}

\subsection{System class}

Let $n,m\in\mathbb{N}$ and $\varepsilon\geq 0$. Consider the class of LTI systems
\begin{equation}
\label{eq:1}
x(t+1)=Ax(t)+Bu(t)+w(t),
\end{equation}
where $x(t)\in\mathbb{R}^n$ is the state, $u(t)\in\mathbb{R}^{m}$ is the input, and $w(t)\in\mathbb{R}^n$ is the process noise satisfying 
\begin{equation}
\label{eq:ass1}
\norm{w(t)}\leq\varepsilon\ \text{ for all }\ t\in\mathbb{Z}_+.
\end{equation}
We identify this class of systems with the set \mbox{$\Sigma\coloneqq\mathbb{R}^{n\times n}\times\mathbb{R}^{n\times m}$}, and we refer to the specific system \eqref{eq:1} and \eqref{eq:ass1} as $(A,B)\in\Sigma$. 

Given $(A,B)\in\Sigma$, we define the \emph{input-state behavior} of \eqref{eq:1} and \eqref{eq:ass1} as
\begin{equation}
\begin{split}
\mathfrak{B}(A,B)\!\coloneqq\!\{(u,x)\!:\!\mathbb{Z}_+\!\!\rightarrow\!\mathbb{R}^{m}\!\times\! \mathbb{R}^{n}\mid \textup{there exists }w\!:\!\mathbb{Z}_+\!\rightarrow\!\mathbb{R}^{n}\\
\textup{such that }\eqref{eq:1}\textup{ and }\eqref{eq:ass1}\textup{ hold}\}.
\end{split}
\end{equation}
In addition, we define the \emph{$k$-restricted input-state behavior} 
\begin{equation}
\mathfrak{B}_k(A,B)\coloneqq\set{(u_{[0,k-1]},x_{[0,k]})}{(u,x)\in\mathfrak{B}(A,B)}.
\end{equation}

Input-state data collected from system \eqref{eq:1} within time horizon $T\in\mathbb{N}$ is denoted by
\begin{equation}
\label{eq:data}
\mathcal{D}\coloneqq(u_{[0,T-1]},x_{[0,T]}).
\end{equation}
Given $\mathcal{D}\in\mathbb{R}^{mT}\times\mathbb{R}^{n(T+1)}$, we call a pair of real matrices $(A,B)\in\Sigma$ a \emph{data-consistent system} if $\mathcal{D}\in\mathfrak{B}_T(A,B)$. We define the set of data-consistent systems as
\begin{equation}
\Sigma_\mathcal{D}\coloneqq\set{(A,B)\in\Sigma}{\mathcal{D}\in\mathfrak{B}_T(A,B)}.
\end{equation}

\section{Problem Formulation}
\label{sec:prob}

Consider the \emph{true} system $(A_\text{true},B_\text{true})\in\Sigma$. We assume that this system is \emph{unknown}, but satisfies
\begin{equation}
(A_\text{true},B_\text{true})\in\Sigma_\text{pk},
\end{equation}
where $\Sigma_\text{pk}\subseteq\Sigma$ is a set that captures our \emph{prior knowledge} of the true system. 

Given data $\mathcal{D}\in\mathfrak{B}_T(A_\text{true},B_\text{true})$ and the prior knowledge set $\Sigma_\text{pk}$, the available information about the true system is
\begin{equation}
(A_\text{true},B_\text{true})\in\Sigma_\mathcal{D}\cap\Sigma_\text{pk}.
\end{equation}
System identification aims at finding an estimate \mbox{$(\hat{A},\hat{B})\in\Sigma$} of the true system using the data and prior knowledge. In case $\Sigma_\mathcal{D}\cap\Sigma_\text{pk}$ is bounded, the \emph{worst-case error} for this estimation is
\begin{equation}
e(\hat{A},\hat{B})\coloneqq\sup_{(A,B)\in\Sigma_\mathcal{D}\cap\Sigma_\text{pk}}\norm{\begin{bmatrix}
A-\hat{A} & B-\hat{B}
\end{bmatrix}}.
\end{equation}
One approach to obtain an estimate $(\hat{A},\hat{B})$ is \mbox{set-membership identification} \cite{milanese1991optimal,li2024learning,shakouri2025chebyshev}, which aims at making the worst-case error as small as possible. This can be achieved by taking $(\hat{A},\hat{B})$ as a Chebyshev center of $\Sigma_\mathcal{D}\cap\Sigma_\text{pk}$, i.e., \mbox{$(\hat{A},\hat{B})\in\cent(\Sigma_\mathcal{D}\cap\Sigma_\text{pk})$}. Such an estimation is called a \emph{best worst-case estimation} of the true system. The worst-case error for this estimation is $e(\hat{A},\hat{B})=\rad(\Sigma_\mathcal{D}\cap\Sigma_\text{pk})$, which represents the \emph{identification accuracy}. In practice, we would like to have a worst-case error that is less than or equal to a given tolerance. It depends on the data $\mathcal{D}$ whether this tolerance can be achieved, leading to the following definition. 


\begin{definition}[Data informativity]
\label{def:inf_rho}
Let $(A,B)\in\Sigma_\text{pk}$, \mbox{$\mathcal{D}\in\mathfrak{B}_T(A,B)$}, and $\rho\geq 0$. We say that $\mathcal{D}$ is \emph{\mbox{$\Sigma_\textup{pk}$--informative} for $\rho$--accuracy identification} if
    \begin{equation}
    \label{eq:rad<rho}
    \rad(\Sigma_\mathcal{D}\cap\Sigma_\text{pk})\leq \rho.
    \end{equation}
In case $\rho=0$, we simply say that $\mathcal{D}$ is $\Sigma_\text{pk}$--\emph{informative for identification}\footnote{$\rad(\Sigma_\mathcal{D}\cap\Sigma_\text{pk})=0$ is equivalent to $\Sigma_\mathcal{D}\cap\Sigma_\text{pk}$ being a singleton, which results in exact identification.}. 
\end{definition}

In this paper, we are interested in \emph{experiment design}. Given a desired $\rho \geq 0$, an interesting problem is to find an input signal that renders the data $\mathcal{D}$, generated by the true system, informative for $\rho$--accuracy identification. The study of this problem is hindered by the fact that the true system, and thus its finite-length trajectories, are unknown. To address this issue, we sharpen the experiment design problem to finding an input signal that, when applied to \emph{any} system in $\Sigma_\text{pk}$, results in informative data. To make this idea precise, we now introduce the following notation. By applying input $u_{[0,T-1]}$ to a system $(A,B)\in\Sigma$, different datasets can be collected depending on the initial state. We define the set of all such datasets as
\begin{equation}
\mathfrak{B}(A,B,u_{[0,T-1]})\coloneqq\set{(\bar{u},\bar{x})\in\mathfrak{B}_T(A,B)}{\bar{u}=u_{[0,T-1]}}.
\end{equation}
Now, we aim at finding an input $u_{[0,T-1]}$ such that for every $(A,B)\in\Sigma_\text{pk}$, all members of $\mathfrak{B}(A,B,u_{[0,T-1]})$ are \mbox{$\Sigma_\text{pk}$--informative} for $\rho$--accuracy identification. Such an input is \emph{universal} in the sense that it works for all systems within $\Sigma_\textup{pk}$. To formalize this, we consider the following definition.

\begin{definition}[Universal inputs]
\label{def:universal_experiment_noisy}
Let $\rho\geq0$. An input \mbox{$u_{[0,T-1]}\in\mathbb{R}^{mT}$} is called \emph{$\Sigma_\textup{pk}$--universal for $\rho$--accuracy identification} if every 
\begin{equation}
\mathcal{D}\in \bigcup_{(A,B)\in\Sigma_\text{pk}}\mathfrak{B}(A,B,u_{[0,T-1]})
\end{equation}
is $\Sigma_\textup{pk}$--informative for $\rho$--accuracy identification. In case \mbox{$\rho=0$}, we simply say that $u_{[0,T-1]}$ is \emph{$\Sigma_\textup{pk}$--universal for identification}.
\end{definition}

A $\Sigma_\text{pk}$--universal input is a solution to the experiment design problem. This is because such an input guarantees that the generated data by the true system can be used for identification with a desired accuracy. Nevertheless, depending on the given prior knowledge, such universal inputs may not exist.  For instance, if $\Sigma_\text{pk}$ is an open set containing an uncontrollable system, then there are no $\Sigma_\text{pk}$--universal inputs for identification, see \cite[Thm. 8]{shakouri2025exp}. Therefore, \emph{universal experiment design} requires suitable prior knowledge. To investigate types of prior knowledge that enable such a design, we define the following notion.  



\begin{definition}
\label{def:enable_universal_experiment_noisy}
Let $\rho\geq 0$. The set $\Sigma_\textup{pk}$ is said to \emph{enable universal experiment design for $\rho$--accuracy identification} if there exists a $\Sigma_\textup{pk}$--universal input for $\rho$--accuracy identification. In case $\rho=0$, we simply say that $\Sigma_\textup{pk}$ \emph{enables universal experiment design for identification}. 
\end{definition}

In this work, we characterize prior knowledge that enables universal experiment design and provide methods for designing inputs that are universal for $\rho$-accuracy identification. 



%% file: UED.tex
\section{Universal Experiment Design With Noise-Free Data}
\label{sec:UED_noisefree}

In the noise-free case, $\varepsilon=0$, a corollary of Willems et al.'s fundamental lemma \cite[Cor. 2]{willems2005note} provides a sufficient condition for an input to be $\Sigma_\text{pk}$--universal for identification provided that all systems in $\Sigma_\text{pk}$ are controllable. We state this result in our terminology as follows. We denote the set of controllable systems in $\Sigma$ by $\Sigma_\text{cont}$.
\begin{proposition}
\label{prop:Willems}
Suppose that \mbox{$\varepsilon=0$} and $\Sigma_\text{pk}\subseteq\Sigma_\text{cont}$. If $u_{[0,T-1]}$ is persistently exciting of order $n+1$, then it is \mbox{$\Sigma_\text{pk}$--universal} for identification. 
\end{proposition}

Persistency of excitation of order $n+1$ requires the number of data samples to satisfy $T\geq nm+n+m$. An interesting question is whether there exist \mbox{$\Sigma_\text{pk}$--universal} inputs that are not persistently exciting of order $n+1$, e.g., of length shorter \mbox{than $nm+n+m$}. To answer this question, we split our study into two parts. First, we consider sets of prior knowledge that are open in $\Sigma$. Later, we extend our results to the general case with arbitrary sets of prior knowledge.

\subsection{Open sets of prior knowledge}
\label{sec:noise-free_open}

In case $\Sigma_\text{pk}$ is an open set, we recall that a rank condition on the data fully characterizes their $\Sigma_\text{pk}$--informativity for identification.
\begin{proposition}[{\cite[Thm. 3]{shakouri2025exp}}]
\label{prop:inf_noise-free_int}
Suppose that $\varepsilon=0$ and $\Sigma_\text{pk}$ is open. Let $\mathcal{D}=(u_{[0,T-1]},x_{[0,T]})\in\mathfrak{B}_T(A_\text{true},B_\text{true})$. Then, the following statements are \emph{equivalent}:
\begin{enumerate}[label=(\alph*),ref=\ref{prop:inf_noise-free_int}(\alph*)]
    \item\label{prop:inf_noise-free_int(a)} $\mathcal{D}$ is $\Sigma_\text{pk}$--informative for identification.
    \item\label{prop:inf_noise-free_int(b)} $\mathcal{D}$ is $\Sigma$--informative for identification
    \item\label{prop:inf_noise-free_int(c)} We have $\rank \begin{bmatrix}
\mathcal{H}_1(x_{[0,T-1]}) \\
\mathcal{H}_1(u_{[0,T-1]})
\end{bmatrix}=n+m$.
\end{enumerate}
\end{proposition}

From Proposition~\ref{prop:inf_noise-free_int} we see that universal experiment design for identification boils down to finding an input by which the resulting data satisfies the rank condition in Proposition~\ref{prop:inf_noise-free_int(c)}. Before going further, we also recall that the set of controllable systems is the largest open set that enables universal experiment design for identification. 

\begin{proposition}[{\cite[Cor. 9]{shakouri2025exp}}]
\label{prop:ctrl_nec}
Suppose that $\varepsilon=0$ and $\Sigma_\text{pk}$ is open. Then, $\Sigma_\textup{pk}$ enables universal experiment design for identification \emph{if and only if} $\Sigma_\textup{pk}\subseteq\Sigma_\text{cont}$.
\end{proposition}


It was shown recently in \cite[Thm. 4]{shakouri2025new} that if $\Sigma_\text{pk}=\Sigma_\text{cont}$, then $\Sigma_\text{pk}$--universality for identification is equivalent to persistency of excitation of order $n+1$. For single-input systems, one can sharpen this result by showing this equivalence not only for $\Sigma_\text{pk}=\Sigma_\text{cont}$, but for \emph{all} open sets of prior knowledge satisfying $\Sigma_\text{pk}\subseteq\Sigma_\text{cont}$.

\begin{theorem}[Single-input systems]
\label{th:single-input-noiseless}
Suppose that $m=1$, $\varepsilon=0$, and $\Sigma_\text{pk}\subseteq\Sigma_\text{cont}$ is open. Then, $u_{[0,T-1]}$ is \mbox{$\Sigma_\textup{pk}$--universal} for identification \emph{if and only if} it is persistently exciting of order $n+1$. 
\end{theorem}

For the multi-input case, however, it is not true in general that all $\Sigma_\text{pk}$--universal inputs are persistently exciting of order $n+1$. This is shown by the following example.

\begin{example}
\label{ex:PE_not_univ}
Consider a set of controllable systems with \mbox{$n=1$} and $m=2$ given by 
\begin{equation}
\Sigma_\text{pk}=\set{(\alpha,\begin{bmatrix}
\beta_1 & \beta_2
\end{bmatrix})}{\alpha\beta_2+ \beta_1\neq 0}.
\end{equation}
Take the input signal as \mbox{$u(0)=u(3)=\begin{bmatrix}
0 & 0
\end{bmatrix}^\top$}, \mbox{$u(1)=\begin{bmatrix}
0 & 1
\end{bmatrix}^\top$}, and \mbox{$u(2)=\begin{bmatrix}
1 & 0
\end{bmatrix}^\top$}. This input is not persistently exciting of order $2$. Applying this input to any member of $\Sigma_\text{pk}$, with initial state $x(0)=x_0$, results in the following input-state data:
\begin{equation}
\begin{bmatrix}
\mathcal{H}_1(x_{[0,T-1]}) \\
\mathcal{H}_1(u_{[0,T-1]})
\end{bmatrix}\!\!=\!\!\begin{bmatrix}
x_0 \!&\! \alpha x_0  \!&\! \alpha^2 x_0\!+\!\beta_2  \!&\! \alpha^3 x_0\!+\!\alpha\beta_2\!+\!\beta_1 \\
0  \!&\! 0  \!&\! 1  \!&\! 0 \\
0  \!&\! 1  \!&\! 0  \!&\! 0
\end{bmatrix}\!.
\end{equation}
We observe that this matrix has full row rank for all $x_0\in\mathbb{R}^n$. Therefore, this input is $\Sigma_\text{pk}$--universal for identification. 
\end{example}

Example~\ref{ex:PE_not_univ} shows that, for multi-input systems, one might be able to find universal inputs of shorter length than persistently exciting ones by incorporating prior knowledge of the system. To study this further, we first introduce some notation. For a pair $(A,B)\in\Sigma$, we define
\begin{equation}
M_{A,B}\coloneqq \begin{bmatrix}
0 & B & AB & \cdots & A^{n-1} B \\
I_m & 0 & 0 & \cdots & 0
\end{bmatrix}, 
\end{equation}
\begin{equation}
D_A\coloneqq \begin{bmatrix}
d_0 & \cdots & d_{n-1} & d_n \\
d_1 & \cdots & d_n & 0 \\
\vdots & \iddots & \vdots & \vdots \\
d_n & \cdots & 0 & 0
\end{bmatrix}\otimes I_m, \hspace{0.25 cm} \textup{and} \hspace{0.25 cm} d_A\coloneqq \begin{bmatrix}
d_0 \\ d_1 \\ \vdots \\ d_n
\end{bmatrix},
\end{equation}
where $d_0,\ldots,d_{n}\in\mathbb{R}$ are the coefficients of the characteristic polynomial of $A$ satisfying $\sum_{i=0}^n d_iA^i=0$ and $d_n=1$. Now, given \emph{any} set of prior knowledge, the following theorem presents a sufficient condition for an input to be $\Sigma_\textup{pk}$--universal for identification. 

\begin{theorem}
\label{th:GM_noiseless}
Suppose that $\varepsilon=0$ and $\Sigma_\text{pk}\subseteq\Sigma_\text{cont}$. Then, the input $u_{[0,T-1]}$ is $\Sigma_\textup{pk}$--universal for identification if 
\begin{equation}
\label{eq:th:GM_noiseless-1}
\rank M_{A,B}D_A\mathcal{H}_{n+1}(u_{[0,T-1]}) = n+m
\end{equation}
for all $(A,B)\in\Sigma_\textup{pk}$.  
\end{theorem}

Under the hypotheses of Theorem \ref{th:GM_noiseless}, the matrix $M_{A,B}$ has full row rank because of the controllability. Moreover, $D_A$ is nonsingular by definition. Hence, it is easy to see that if the input is persistently exciting of order $n+1$, then \eqref{eq:th:GM_noiseless-1} is satisfied. Nevertheless, condition \eqref{eq:th:GM_noiseless-1} on the input signal is less conservative than persistency of excitation of order $n+1$. This is because \eqref{eq:th:GM_noiseless-1} can be satisfied without $\mathcal{H}_{n+1}(u_{[0,T-1]}) $ having full row rank. In fact, condition \eqref{eq:th:GM_noiseless-1} needs the number of data samples to satisfy
\begin{equation}
\label{eq:lower_bound}
T\geq2n+m,
\end{equation}
which improves the least number of data samples $nm+n+m$ required for the persistency of excitation of order $n+1$. The following example shows that, for some sets of prior knowledge, one can design universal inputs using Theorem~\ref{th:GM_noiseless} such that the lower bound in~\eqref{eq:lower_bound} is attained.

\begin{example}
Consider the set $\Sigma_\text{pk}$ and the input signal $u_{[0,3]}$ as in Example~\ref{ex:PE_not_univ}. Here, we use Theorem~\ref{th:GM_noiseless} to show that $u_{[0,3]}$ is $\Sigma_\text{pk}$--universal for identification. Observe that 
\begin{equation}
M_{A,B}D_A\mathcal{H}_{2}(u_{[0,3]})=\begin{bmatrix}
0 & \beta_2 & \beta_1 \\
0 & 1  & -\alpha \\
1 & -\alpha &  0
\end{bmatrix}.
\end{equation}
Thus, we have \mbox{$\det M_{A,B}D_A\mathcal{H}_{2}(u_{[0,3]})=- \alpha\beta_2-\beta_1\neq 0$}. Hence, \eqref{eq:th:GM_noiseless-1} holds for all $(A,B)\in\Sigma_\text{pk}$. Therefore, $u_{[0.3]}$ is $\Sigma_\text{pk}$--universal for identification. We emphasize that this input is not persistently exciting of order $2$ and the number of data samples is equal to the lower bound in \eqref{eq:lower_bound}.
\end{example}

Theorem~\ref{th:GM_noiseless} shows that, depending on the prior knowledge, one might be able to find universal inputs that are not persistently exciting of order $n+1$, leading to a shorter experiment. Nevertheless, for some $\Sigma_\text{pk}\subseteq\Sigma_\text{cont}$, persistently exciting inputs of order $n+1$ are the only $\Sigma_\text{pk}$--universal ones. As mentioned before, one such case corresponds to $\Sigma_\text{pk}=\Sigma_\text{cont}$, see \cite[Thm. 4]{shakouri2025new}. Now, the following theorem shows that for any open set of prior knowledge including a norm bound on $B_\text{true}$, universality and persistency of excitation of order $n+1$ are equivalent. 

\begin{theorem}
\label{th:PE_nec_ball}
Suppose that $\varepsilon=0$ and
\begin{equation}
\Sigma_\text{pk}=\set{(A,B)\in\Sigma_\text{cont}}{A\in\mathcal{A}_\text{pk},\norm{B}< \beta},
\end{equation}
where $\beta>0$ and $\mathcal{A}_\text{pk}\subseteq\mathbb{R}^{n\times n}$ is open. Then, $u_{[0,T-1]}$ is $\Sigma_\textup{pk}$--universal for identification \emph{if and only if} it is persistently exciting of order $n+1$. 
\end{theorem}

\subsection{Arbitrary sets of prior knowledge}
\label{sec:noise-free_arbitrary}

In this section, we consider the general case where the prior knowledge set $\Sigma_\text{pk}$ need not be open. For arbitrary prior knowledge, the following proposition provides a necessary and sufficient condition for $\Sigma_\text{pk}$--informativity for identification.

\begin{proposition}
\label{prop:data_best}
Suppose that $\varepsilon=0$. Let $T\in\mathbb{N}$, \linebreak$\mathcal{D}\!=\!(u_{[0,T-1]},x_{[0,T]})\!\in\!\mathfrak{B}_T(A_\text{true},\!B_\text{true})$ and $(\hat{A},\!\hat{B})\!\in\!\Sigma_\text{pk}\cap\Sigma_\mathcal{D}$. Then, $\mathcal{D}$ is $\Sigma_\text{pk}$--informative for identification \emph{if and only if}
\begin{equation}
\label{eq:rank_data_gen}
\begin{bmatrix}
\hat{A}-A & \hat{B}-B
\end{bmatrix}\begin{bmatrix}
\mathcal{H}_1(x_{[0,T-1]}) \\
\mathcal{H}_1(u_{[0,T-1]})
\end{bmatrix}\neq 0,
\end{equation}
for all $(A,B)\in\Sigma_\text{pk}\backslash\{(\hat{A},\hat{B})\}$.
\end{proposition}

Proposition~\ref{prop:data_best} generalizes the result of Proposition~\ref{prop:inf_noise-free_int} to the case of arbitrary prior knowledge. The condition in Proposition~\ref{prop:data_best} depends on $(\hat{A},\hat{B})$ that needs to be taken from $\Sigma_\text{pk}\cap\Sigma_\mathcal{D}$. To bypass this requirement, the following corollary provides a sufficient condition for $\Sigma_\text{pk}$--informativity for identification, which does not depend on $(\hat{A},\hat{B})$. 

\begin{corollary}
\label{cor:inf_noise-free_general}
Suppose that $\varepsilon=0$. Let $T\in\mathbb{N}$ and  \linebreak $\mathcal{D}=(u_{[0,T-1]},x_{[0,T]})\in\mathfrak{B}_T(A_\text{true},B_\text{true})$. Then,  $\mathcal{D}$ is \mbox{$\Sigma_\text{pk}$--informative} for identification if $\begin{bmatrix}
F & G
\end{bmatrix}\begin{bmatrix}
\mathcal{H}_1(x_{[0,T-1]}) \\
\mathcal{H}_1(u_{[0,T-1]})
\end{bmatrix}$ is nonzero for all nonzero $(F,G)\in\Sigma_\text{pk}-\Sigma_\text{pk}$.
\end{corollary}\vspace{0.25 cm}

Even though the condition presented in Corollary~\ref{cor:inf_noise-free_general} is only sufficient, it is less conservative than the rank condition in Proposition~\ref{prop:inf_noise-free_int(c)}. Now, we turn our attention to finding $\Sigma_\text{pk}$--universal inputs for arbitrary $\Sigma_\text{pk}$. To study this, we introduce some notation. Let $k\in[1,n]$ and define
\begin{equation}
\label{eq:Sigma^{(k)}}
\Sigma^{(k)}\coloneqq\set{(A,B)\in\Sigma}{\deg(\mu_A)\leq k}.
\end{equation}
The set $\Sigma^{(k)}$ contains all matrix pairs $(A,B)\in\Sigma$ with the property that the degree of the minimal polynomial of $A$ is less than or equal to $k$. We note that $\Sigma^{(n)}=\Sigma$. For $(A,B)\in\Sigma^{(k)}$, we also define
\begin{subequations}
\label{eq:definitions}
\begin{align}
\label{eq:M_A,B^k}
M_{A,B}^{(k)}&\coloneqq \begin{bmatrix}
0 & B & AB & \cdots & A^{k-1} B \\
I_m & 0 & 0 & \cdots & 0
\end{bmatrix}, \\
\label{eq:d_A^k}
d_A^{(k)}&\coloneqq \begin{bmatrix}
d_0 & \cdots & d_{k-1} & d_k
\end{bmatrix}^\top,\\
\label{eq:D_A^k}
D_{A}^{(k)}&\coloneqq \begin{bmatrix}
d_0 & \cdots & d_{k-1} & d_k \\
d_1 & \cdots & d_k & 0 \\
\vdots & \iddots & \vdots & \vdots \\
d_k & \cdots & 0 & 0
\end{bmatrix}\otimes I_m.
\end{align}
\end{subequations}
where the entries $d_i$, $i\in[0,k]$, satisfy $\sum_{i=0}^{k}d_iA^i=0$, and $d_k=1$. We note that $d_A^{(n)}=d_A$, $D_A^{(n)}=D_A$ and $M_{A,B}^{(n)}=M_{A,B}$. Now, the following theorem presents a sufficient condition for an input to be $\Sigma_\text{pk}$--universal for identification, where $\Sigma_\text{pk}$ is an arbitrary set. 

\begin{theorem}
\label{th:GM_noiseless_generalized}
Suppose that $\varepsilon=0$. Let $k\in[1,n]$ be such that $\Sigma_\text{pk}\subseteq\Sigma^{(k)}$. Then, the input $u_{[0,T-1]}$ is $\Sigma_\textup{pk}$--universal for identification if 
\begin{equation}
\label{eq:th:GM_noiseless_generalized-1}
\begin{bmatrix}
F & G
\end{bmatrix}M_{A,B}^{(k)}D_A^{(k)}\mathcal{H}_{k+1}(u_{[0,T-1]})\neq 0
\end{equation}
for all nonzero $(F,G)\in\Sigma_\text{pk}-\Sigma_\text{pk}$ and all $(A,B)\in\Sigma_\text{pk}$.
\end{theorem}

Compared to Theorem~\ref{th:GM_noiseless}, Theorem~\ref{th:GM_noiseless_generalized} incorporates a broader class of prior knowledge. For instance, this theorem can take into account exact knowledge of the entries of $(A_\text{true},B_\text{true})$, which cannot be reflected in an open $\Sigma_\text{pk}$. Moreover, Theorem~\ref{th:GM_noiseless_generalized} incorporates prior knowledge on the degree of the minimal polynomial\footnote{This type of prior knowledge is relevant for systems with a known structure, e.g., a network of identical systems where $A_\text{true}$ is block diagonal, see \cite[Cor. 2]{yu2021controllability}.} of $A_\text{true}$. Such prior knowledge cannot be represented by an open $\Sigma_\text{pk}$. This is because an open $\Sigma_\text{pk}$ satisfies $\Sigma_\text{pk}\subseteq\Sigma^{(k)}$ only if $k=n$, which corresponds to the case where no knowledge of $\deg(\mu_{A_\text{true}})$ is available. 

The condition in Theorem~\ref{th:GM_noiseless_generalized} requires the number of samples to satisfy $T\geq k+1$, where $k\in[1,n]$ is an upper bound for $\deg(\mu_{A_\text{true}})$. The following example shows that, for some sets of prior knowledge, one can design a universal input with exactly $k+1$ samples.

\begin{example}
\label{ex:3}
Consider the set of prior knowledge as
\begin{equation}
\Sigma_\text{pk}=\set{\left(\begin{bmatrix}
0 & \alpha & \alpha \\
0 & 0        & 0 \\
0 & 0        & 0
\end{bmatrix},\begin{bmatrix}
0 \\ 0 \\ \beta
\end{bmatrix}\right)}{\alpha\in\mathbb{R},\beta\neq 0}.
\end{equation}
Every $(A,B)\in\Sigma_\text{pk}$ satisfies $A^2=0$. Hence, we take $k=2$ with $d_2=1$ and $d_1=d_0=0$ to have $\sum_{i=0}^2 d_iA^i=0$. Let $e_{i}$ be a vector with all entries equal to zero but the $i$th entry equal to $1$. We observe that every \mbox{$(F,G)\in\Sigma_\text{pk}-\Sigma_\text{pk}$} can be written as \mbox{$F=f\begin{bmatrix}
0 \!&\! e_1 \!&\! e_1
\end{bmatrix}$} and $G=ge_3$
for some $f,g\in\mathbb{R}$. Verify that \mbox{$\begin{bmatrix}
F \!\!&\!\! G
\end{bmatrix}M_{A,B}^{(2)}D_A^{(2)}=\begin{bmatrix}
0 \!&\! \beta fe_1 \!&\! ge_3
\end{bmatrix}$}. Take \mbox{$T=3$} and the input as $u(0)=0$, \mbox{$u(1)=u(2)=1$}. Observe that condition \eqref{th:GM_noiseless_generalized} reduces to \mbox{$\begin{bmatrix}
F & G
\end{bmatrix}M_{A,B}^{(2)}D_A^{(2)}\mathcal{H}_{3}(u_{[0,2]})=\begin{bmatrix}
\beta f & 0 & g
\end{bmatrix}^\top$}, which is zero only if $f=g=0$. Therefore, this input is \mbox{$\Sigma_\text{pk}$--universal} for identification.
\end{example}

\begin{remark}
\label{rem:uncont}
Recall from Proposition~\ref{prop:ctrl_nec} that if the set of prior knowledge is open and contains an uncontrollable pair, then it does not enable universal experiment design for identification. Interestingly, this is not necessarily true if the prior knowledge is not open. For instance, the set $\Sigma_\text{pk}$ in Example~\ref{ex:3} enables universal experiment design for identification, while all its members are uncontrollable. 
\end{remark}


The condition in Theorem~\ref{th:GM_noiseless_generalized} depends on the elements of $\Sigma_\text{pk}$ and $\Sigma_\text{pk}-\Sigma_\text{pk}$. In case $\Sigma_\text{pk}\subseteq\Sigma_\text{cont}$, a more conservative condition than that of Theorem~\ref{th:GM_noiseless_generalized} relaxes the dependency on $\Sigma_\text{pk}-\Sigma_\text{pk}$. This is discussed in the following corollary. 

\begin{corollary}
\label{cor:GM_noiseless_generalized_prime}
Suppose that $\varepsilon=0$ and $\Sigma_\text{pk}\subseteq\Sigma_\text{cont}$. Let $k\in[1,n]$ be such that $\Sigma_\text{pk}\subseteq\Sigma^{(k)}$. Then, the input $u_{[0,T-1]}$ is $\Sigma_\textup{pk}$--universal for identification if 
\begin{equation}
\label{eq:cor:GM_noiseless_generalized-1}
\rank M_{A,B}^{(k)}D_A^{(k)}\mathcal{H}_{k+1}(u_{[0,T-1]})=n+m
\end{equation}
for all $(A,B)\in\Sigma_\text{pk}$.
\end{corollary}

It is evident that, under the hypothesis of Corollary~\ref{cor:GM_noiseless_generalized_prime}, if $\mathcal{H}_{k+1}(u_{[0,T-1]})$ has full row rank, then \eqref{eq:cor:GM_noiseless_generalized-1} holds for all $(A,B)\in\Sigma_\text{pk}$. This leads to another extension of Willems et al.'s fundamental lemma, stated next, which was initially observed in \cite[Cor. 1]{yu2021controllability}. 

\begin{corollary}
\label{cor:Willems_generalized}
Suppose that $\varepsilon=0$ and $\Sigma_\text{pk}\subseteq\Sigma_\text{cont}$. Let \mbox{$k\in[1,n]$} be such that $\Sigma_\text{pk}\subseteq\Sigma^{(k)}$. If $u_{[0,T-1]}$ is persistently exciting of order $k+1$, then it is $\Sigma_\text{pk}$--universal for identification.
\end{corollary}


We note that in case $\Sigma_\text{pk}\subseteq\Sigma_\text{cont}$ is open, the statements of Corollary~\ref{cor:GM_noiseless_generalized_prime} and Corollary~\ref{cor:Willems_generalized} coincide with that of Theorem~\ref{th:GM_noiseless} and Proposition~\ref{prop:Willems}, respectively. 

%% file: UED_noisy.tex
\section{Universal Experiment Design
With Noisy Data}
\label{sec:i-o gain}

This section considers universal experiment design in the presence of noise. In this setting, universal experiment design for  (exact) identification requires stringent conditions on the prior knowledge set. Nevertheless, finding universal inputs for $\rho$--accuracy identification, with given \mbox{$\rho>0$}, can be feasible under mild conditions on the prior knowledge set. As such, we first study universal experiment design for $\rho$--accuracy identification in Section~\ref{sub:UED_noisy_rho>0}. Universal experiment design for identification will be studied in detail in Section~\ref{sub:UED_noisy_rho=0}.  

\subsection{Experiment design for $\rho$--accuracy identification}
\label{sub:UED_noisy_rho>0}

The first step towards designing a $\Sigma_\text{pk}$--universal input for $\rho$--accuracy identification is to find conditions under which the data $\mathcal{D}$ is $\Sigma_\text{pk}$--informative for $\rho$--accuracy identification, i.e., $\rad(\Sigma_\mathcal{D}\cap\Sigma_\text{pk})\leq \rho$. Computing the exact value of \mbox{$\rad(\Sigma_\mathcal{D}\cap\Sigma_\text{pk})$}, however, can be a hard problem\footnote{Even for some convex sets, such as the intersection of balls, computing the Chebyshev radius is generally NP-hard, see \cite[Thm. 2]{xia2021chebyshev}}. Consequently, we focus on a conservative condition for \mbox{$\Sigma_\text{pk}$--informativity}, namely, $\rad\Sigma_\mathcal{D}\leq \rho$. Since \mbox{$\rad(\Sigma_\mathcal{D}\cap\Sigma_\text{pk})\leq\rad\Sigma_\mathcal{D}$}, we have that $\mathcal{D}$ is \mbox{$\Sigma_\text{pk}$--informative} for \mbox{$\rho$--accuracy} identification if $\rad\Sigma_\mathcal{D}\leq \rho$. 

We investigate how $\rad\Sigma_\mathcal{D}$ is related to the data $\mathcal{D}$. Given \mbox{$\mathcal{D}=(u_{[0,T-1]},x_{[0,T]})$}, we define
\begin{equation}
\sigma_\mathcal{D}\coloneqq\sigma_{n+m}\left(\begin{bmatrix}
    \mathcal{H}_1(x_{[0,T-1]}) \\
    \mathcal{H}_1(u_{[0,T-1]})
    \end{bmatrix}\right).
\end{equation}
We recall from Proposition~\ref{prop:inf_noise-free_int} that, in case $\varepsilon=0$, if \mbox{$\mathcal{D}\in\mathfrak{B}_T(A_\text{true},B_\text{true})$} satisfies $\sigma_\mathcal{D}>0$, then $\rad\Sigma_\mathcal{D}=0$, and hence the data is informative for identification. This, however, is not true when $\varepsilon>0$. In the noisy setting, $\sigma_\mathcal{D}>0$ only implies that $\Sigma_\mathcal{D}$ is bounded \cite[Prop. 14(b)]{shakouri2025chebyshev}. We note that, yet there is no general closed-form expression for $\rad\Sigma_\mathcal{D}$ in the literature\footnote{In special cases, there exist algorithmic methods that can approximate $\rad\Sigma_\mathcal{D}$, see, e.g., \cite{beck2007regularization}}. Nevertheless, one can find upper and lower bounds for $\rad\Sigma_\mathcal{D}$ in terms of $\sigma_\mathcal{D}$ and $\varepsilon$, which is shown by the following lemma. 

\begin{lemma}
\label{lem:noisy_rad_lower}
Let $\varepsilon\geq 0$, $T\in\mathbb{N}$, and $(A,B)\in\Sigma$. Then, the following statements hold:
\begin{enumerate}[label=(\alph*),ref=\ref{lem:noisy_rad_lower}(\alph*)]
    \item\label{lem:noisy_rad_lower(a)} Let $\mathcal{D}\in\mathfrak{B}_T(A,B)$. Then, $\Sigma_\mathcal{D}$ is bounded \emph{if and only if} $\sigma_\mathcal{D}>0$. Moreover, if $\sigma_\mathcal{D}>0$, then $\rad\Sigma_\mathcal{D}\leq \frac{\sqrt{T}\varepsilon}{\sigma_\mathcal{D}}$.
    \item\label{lem:noisy_rad_lower(b)} For every $u_{[0,T-1]}$ there exists $\mathcal{D}\in\mathfrak{B}(A,B,u_{[0,T-1]})$ such that either $\Sigma_\mathcal{D}$ is unbounded, or $\sigma_\mathcal{D}>0$ and $\rad\Sigma_\mathcal{D}\geq \frac{\varepsilon}{\sigma_\mathcal{D}}$.
\end{enumerate}
\end{lemma}\vspace{0.1 cm}

Lemma~\ref{lem:noisy_rad_lower(a)} presents a way to over-approximate $\rad\Sigma_\mathcal{D}$ as long as the dataset satisfies $\sigma_\mathcal{D}>0$. The lower bound provided in Lemma~\ref{lem:noisy_rad_lower(b)}, however, only holds for some datasets. Both approximations are proportional to a \emph{signal-to-noise ratio} $\frac{\varepsilon}{\sigma_\mathcal{D}}$. 

Now, for given $\rho>0$, we aim at finding inputs such that, when applied to a system, yield data satisfying \mbox{$\rad\Sigma_\mathcal{D}\leq \rho$}. The following lemma provides a sufficient condition for finding such inputs. For $A\in\mathbb{R}^{n\times n}$, we recall that $\deg(\mu_A)$ refers to the degree of its minimal polynomial. We also define
\begin{equation}
\label{eq:M_A^k}
M_A^{(k)}\coloneqq \begin{bmatrix}
0 & I_n & A & \cdots & A^{k-1} \\
I_n & 0 & 0 & \cdots & 0
\end{bmatrix}.
\end{equation}

\begin{lemma}
\label{lem:open_design_noisy}
Suppose that $\varepsilon>0$. Let \mbox{$(A,B)\in\Sigma_\text{cont}$}, \mbox{$k\in[\deg(\mu_A),n]$}, and $\rho>0$. Then, $\rad\Sigma_\mathcal{D}\leq \rho$ for all \mbox{$\mathcal{D}\in\mathfrak{B}(A,B,u_{[0,T-1]})$} if 
\begin{equation}
\label{eq:th:open_design_noisy_main}
\begin{split}
&\sigma_{n+m}(M^{(k)}_{A,B}D_{\!A}^{(k)}\mathcal{H}_{k+1}(u_{[0,T-1]}))\\
&\geq\varepsilon\sqrt{T(k+1)}\|d_A^{(k)}\|_1\|M_{A}^{(k)}\|+\frac{\varepsilon}{\rho}\sqrt{T}\|d_A^{(k)}\|_1.
\end{split}
\end{equation}
\end{lemma}

Condition \eqref{eq:th:open_design_noisy_main} requires the smallest singular value of $M^{(k)}_{A,B}D_{\!A}^{(k)}\mathcal{H}_{k+1}(u_{[0,T-1]})$ to be larger than or equal to the sum of two terms $\varepsilon\sqrt{T(k+1)}\|d_A^{(k)}\|_1\|M_{A}^{(k)}\|$ and $\frac{\varepsilon}{\rho}\sqrt{T}\|d_A^{(k)}\|_1$. The first term guarantees that $\Sigma_\mathcal{D}$ is bounded, i.e., if $u_{[0,T-1]}$ satisfies
\begin{equation}
\label{eq:th:open_design_noisy_main(bounded)}
\begin{split}
&\sigma_{n+m}(M^{(k)}_{A,B}D_{\!A}^{(k)}\mathcal{H}_{k+1}(u_{[0,T-1]}))\\
&>\varepsilon\sqrt{T(k+1)}\|d_A^{(k)}\|_1\|M_{A}^{(k)}\|,
\end{split}
\end{equation}
then $\Sigma_\mathcal{D}$ is bounded for all \mbox{$\mathcal{D}\in\mathfrak{B}(A,B,u_{[0,T-1]})$}. This is because \eqref{eq:th:open_design_noisy_main(bounded)} implies that \eqref{eq:th:open_design_noisy_main} holds for a sufficiently large $\rho>0$. In case $\varepsilon=0$, \eqref{eq:th:open_design_noisy_main(bounded)} boils down to~\eqref{eq:cor:GM_noiseless_generalized-1}. 

Lemma~\ref{lem:open_design_noisy} now leads to the following sufficient condition for an input to be $\Sigma_\text{pk}$--universal for \mbox{$\rho$--accuracy} identification. 

\begin{theorem}
\label{th:open_design_noisy}
Suppose that $\varepsilon>0$ and $\Sigma_\text{pk}\subseteq\Sigma_\text{cont}$. Let $\rho>0$ and $k\in[1,n]$ be such that $\Sigma_\text{pk}\subset\Sigma^{(k)}$. Then, $u_{[0,T-1]}$ is $\Sigma_\text{pk}$--universal for $\rho$--accuracy identification if \eqref{eq:th:open_design_noisy_main} holds for all $(A,B)\in\Sigma_\text{pk}$. 
\end{theorem}

Theorem~\ref{th:open_design_noisy} requires all systems complying with the prior knowledge to be controllable. This is because \eqref{eq:th:open_design_noisy_main} cannot be satisfied for an uncontrollable system. Nevertheless, satisfying \eqref{eq:th:open_design_noisy_main} for all $(A,B)\in\Sigma_\text{pk}$ requires a stronger condition on $\Sigma_\text{pk}$ than $\Sigma_\text{pk}\subseteq\Sigma_\text{cont}$. To further investigate this, we define
\begin{equation}
\label{eq:gamma_def}
\gamma_\textup{pk}\!\coloneqq\! \inf\set{\sigma_n\left(\begin{bmatrix}
B \!\!&\!\! AB \!\!&\!\! \cdots \!\!&\!\! A^{n-1}B
\end{bmatrix}\right)}{(A,B)\!\in\!\Sigma_\textup{pk}}\!.
\end{equation}
This can be viewed as a measure of controllability for systems within the prior knowledge set. The next theorem characterizes all bounded sets of prior knowledge for which Theorem~\ref{th:open_design_noisy} leads to a universal input. Consequently, it provides a sufficient condition under which universal experiment design for \mbox{$\rho$--accuracy} identification is feasible. 

\begin{theorem}
\label{th:open_enabling_noisy}
Suppose that $\Sigma_\textup{pk}$ is bounded. Let $\rho>0$ and $k\in[1,n]$ be such that $\Sigma_\text{pk}\subset\Sigma^{(k)}$. Then, the following statements hold:
\begin{enumerate}[label=(\alph*),ref=\ref{th:open_enabling_noisy}(\alph*)]
    \item\label{th:open_enabling_noisy(a)} There exists $T\in\mathbb{N}$ and $u_{[0,T-1]}\in\mathbb{R}^{mT}$ such that \eqref{eq:th:open_design_noisy_main} holds for all $(A,B)\in\Sigma_\text{pk}$ \emph{if and only if} $\gamma_\textup{pk}>0$.
    \item\label{th:open_enabling_noisy(b)} If $\gamma_\textup{pk}>0$, then $\Sigma_\text{pk}$ enables universal experiment design for $\rho$--accuracy identification.
\end{enumerate}
\end{theorem}

Condition $\gamma_\textup{pk}>0$ in Theorem~\ref{th:open_enabling_noisy(b)} is not necessary in general, i.e., there exist bounded sets of prior knowledge with $\gamma_\textup{pk}=0$ that enable universal experiment design for \mbox{$\rho$--accuracy} identification. To elaborate on this, we first present the following example.

\begin{example}
\label{ex:help0}
Let $n=m=1$ and consider the set of prior knowledge given by
\begin{equation}
\label{eq:pk_ex:help0}
\Sigma_\text{pk}=\set{(A,B)\in\mathbb{R}\times\mathbb{R}}{0\leq|A|\leq|B|\leq1}.
\end{equation}
For this set, we have $\gamma_\text{pk}=0$ since $(0,0)\in\Sigma_\text{pk}$ is not controllable. Hence, it is evident that \eqref{eq:th:open_design_noisy_main} cannot be satisfied for all $(A,B)\in\Sigma_\text{pk}$. However, one can satisfy \eqref{eq:th:open_design_noisy_main} for all $(A,B)\in\Sigma_\text{pk}$ excluding a small neighborhood of the origin. More precisely, for $\rho>0$, one can partition $\Sigma_\text{pk}$ into two subsets $\Sigma^\prime_\text{pk}$ and $\Sigma_\text{pk}^{\prime\prime}$ (see Fig. \ref{fig:1}) as
\begin{subequations}
\label{eq:pk_ex:help0_partition}
\begin{align}
\label{eq:pk_ex:help0_partition(a)}
\Sigma^\prime_\text{pk}&=\set{(A,B)\in\Sigma_\text{pk}}{\tfrac{\sqrt{2}}{2}\rho\leq|B|\leq1},\\ \label{eq:pk_ex:help0_partition(b)}
\Sigma_\text{pk}^{\prime\prime}&=\set{(A,B)\in\Sigma_\text{pk}}{|B|\leq\tfrac{\sqrt{2}}{2}\rho}.
\end{align}
\end{subequations}
These sets satisfy \mbox{$\Sigma_\text{pk}^\prime\cup\Sigma_\text{pk}^{\prime\prime}=\Sigma_\text{pk}$} and $\rad\Sigma_\text{pk}^{\prime\prime}\leq \rho$. In view of Theorem~\ref{th:open_enabling_noisy}, there exists $u_{[0,T-1]}$ with the property that, when applied to any $(A,B)\in\Sigma^\prime_\text{pk}$, yields data satisfying \mbox{$\rad\Sigma_\mathcal{D}\leq\rho$}. Now, suppose that we apply this input to some $(A,B)\in\Sigma_\text{pk}$. In case the resulting data $\mathcal{D}$ satisfies \mbox{$\rad\Sigma_\mathcal{D}\leq\rho$}, obviously, $\mathcal{D}$ is $\Sigma_\text{pk}$--informative data for \mbox{$\rho$--accuracy} identification. On the other hand, if $\rad\Sigma_\mathcal{D}>\rho$, due to the property of the input signal, we have \mbox{$\Sigma_\mathcal{D}\cap\Sigma^\prime_\text{pk}=\varnothing$}. Hence, $\Sigma_\mathcal{D}\cap\Sigma_\text{pk}=\Sigma_\mathcal{D}\cap\Sigma_\text{pk}^{\prime\prime}$. Since $\rad\Sigma_\text{pk}^{\prime\prime}\leq \rho$, we have
\begin{equation}
\rad(\Sigma_\mathcal{D}\cap\Sigma_\text{pk})=\rad(\Sigma_\mathcal{D}\cap\Sigma_\text{pk}^{\prime\prime})\leq\rad\Sigma_\text{pk}^{\prime\prime}\leq \rho.
\end{equation}
This implies that, also in this case, $\mathcal{D}$ is $\Sigma_\text{pk}$--informative for \mbox{$\rho$--accuracy} identification. Therefore, $u_{[0,T-1]}$ is \mbox{$\Sigma_\text{pk}$--universal} for \mbox{$\rho$--accuracy} identification. 
\end{example}

Example~\ref{ex:help0} shows that, in case \eqref{eq:th:open_design_noisy_main} cannot be satisfied for all $(A,B)\in\Sigma_\text{pk}$, one may still be able to find universal inputs by partitioning $\Sigma_\text{pk}$ into two subsets with certain properties. We formalize this idea as follows. 

\begin{theorem}
\label{th:helping_in_exp}
Let $\rho\geq 0$ and $\Sigma_\text{pk}^\prime\subseteq\Sigma$. Let $u_{[0,T-1]}$ be such that for every $(A,B)\in\Sigma_\text{pk}^\prime$, all datasets \mbox{$\mathcal{D}\in\mathfrak{B}(A,B,u_{[0,T-1]})$} satisfy $\rad\Sigma_\mathcal{D}\leq \rho$. Then, $u_{[0,T-1]}$ is \mbox{$(\Sigma_\text{pk}^\prime\cup\Sigma_\text{pk}^{\prime\prime})$--universal} for \mbox{$\rho$--accuracy} identification for all $\Sigma_\text{pk}^{\prime\prime}\subseteq\Sigma$ with $\rad\Sigma_\text{pk}^{\prime\prime}\leq \rho$.
\end{theorem}

In view of Theorem~\ref{th:helping_in_exp}, to design an experiment for \mbox{$\rho$--accuracy identification}, it suffices to only look at a certain subset of $\Sigma_\text{pk}$ rather than the entire set. This result, along with Theorem~\ref{th:open_design_noisy}, may lead to universal inputs in case Theorem~\ref{th:open_design_noisy} by itself does not provide a solution. This has applications when the adversarial elements of $\Sigma_\text{pk}$, for which \eqref{eq:th:open_design_noisy_main} does not hold or is difficult to satisfy, lie within a sufficiently small subset of $\Sigma_\text{pk}$. We illustrate this by revisiting Example~\ref{ex:help0} in the following. 

\begin{example}
\label{ex:help}
Consider $\Sigma_\text{pk}$ as in \eqref{eq:pk_ex:help0}. In case $\rho\geq \sqrt{2}$, we have \mbox{$\rad(\Sigma_\mathcal{D}\cap\Sigma_\text{pk})\leq\rad \Sigma_\text{pk}=\sqrt{2}\leq \rho$}. As such, there is no need for data since the prior knowledge itself already provides the required accuracy. Hence, in this case, any input is \mbox{$\Sigma_\text{pk}$--universal} for \mbox{$\rho$--accuracy} identification. Now, suppose that $\rho< \sqrt{2}$. Take $\Sigma^\prime_\text{pk}$ and $\Sigma_\text{pk}^{\prime\prime}$ as in \eqref{eq:pk_ex:help0_partition(a)} and \eqref{eq:pk_ex:help0_partition(b)}. Take $u_{[0,T-1]}$ such that \eqref{eq:th:open_design_noisy_main} holds for all $(A,B)\in\Sigma_\text{pk}^\prime$. Since \mbox{$\Sigma_\text{pk}^\prime\cup\Sigma_\text{pk}^{\prime\prime}=\Sigma_\text{pk}$} and $\rad\Sigma_\text{pk}^{\prime\prime}\leq \rho$, it follows from Theorem~\ref{th:helping_in_exp} that $u_{[0,T-1]}$ is $\Sigma_\text{pk}$--universal for $\rho$--accuracy identification. For instance, for $\rho=1$, one can take \mbox{$T=3$}, $u(0)=u(2)=0$, and $|u(1)| \geq 8.86\varepsilon$, and verify that $u_{[0,2]}$ satisfies \eqref{eq:th:open_design_noisy_main} for all $(A,B)\in\Sigma_\text{pk}^\prime$.
\end{example}

\begin{figure}[h]
    \centering
    \begin{tikzpicture}[scale=1]
    \filldraw[red!10] (-1cm,-1cm)--(-1cm,1cm)--(0,0)--cycle;
    \filldraw[red!10] (1cm,1cm)--(1cm,-1cm)--(0,0)--cycle;
    \filldraw[blue!10] (-0.5cm,-0.5cm)--(-0.5cm,0.5cm)--(0,0)--cycle;
    \filldraw[blue!10] (0.5cm,0.5cm)--(0.5cm,-0.5cm)--(0,0)--cycle;
    \draw[rotate=45,dashed] (0,-1.412cm)--(0,1.412cm);
    \draw[rotate=-45,dashed] (0,-1.412cm)--(0,1.412cm);
    \draw[dashed] (1cm,-1cm)--(1cm,1cm);
    \draw[dashed] (-1cm,-1cm)--(-1cm,1cm);
    \draw[dashed] (0.5cm,-0.5cm)--(0.5cm,0.5cm);
    \draw[dashed] (-0.5cm,-0.5cm)--(-0.5cm,0.5cm);
    \draw[->] (0,-1cm)--(0,1cm);
    \draw[->] (-1.5cm,0)--(1.5cm,0);
    \draw[rotate=-45,|<->|] (-0.1,0)--(-0.1,0.7cm);
    \node[align=center] at (1.7cm,0) {\footnotesize $B$};
    \node[align=center] at (0cm,1.2cm) {\footnotesize $A$};
    \node[align=center] at (0.15cm,0.43cm) {\footnotesize $\rho$};
    \end{tikzpicture}
    \caption{Sets $\Sigma_\text{pk}^\prime$ and $\Sigma_\text{pk}^{\prime\prime}$, for Examples~\ref{ex:help0} and~\ref{ex:help}, shown by the red and blue areas, respectively.}
    \label{fig:1}
\end{figure}
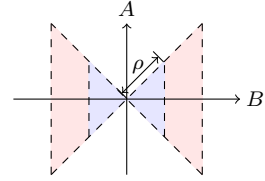

\subsection{Experiment design for exact identification}
\label{sub:UED_noisy_rho=0}

In this section, we consider universal experiment design for identification. We begin by defining the following notion:
\begin{equation}
\delta_\text{pk}\coloneqq\inf\set{\norm{\begin{bmatrix}
F & G
\end{bmatrix}}\neq0}{(F,G)\in\Sigma_\text{pk}-\Sigma_\text{pk}}.
\end{equation}
In case $\delta_\text{pk}>0$, we say that $\Sigma_\text{pk}$ is \emph{uniformly discrete}\footnote{Uniformly discrete sets of prior knowledge appear in certain applications such as population dynamics or multi-agent systems with unknown topology, where the unknown parameters of the system are integers.}, see \cite[p. 36]{beer1993topologies}. This means that every two distinct elements of $\Sigma_\text{pk}$ are separated at least by the distance $\delta_\text{pk}$. This property is closely related to the feasibility of universal experiment design for identification in the noisy setting. 

\begin{proposition}
\label{prop:impossible_noisy}
Suppose that $\varepsilon>0$. If $\Sigma_\text{pk}$ enables universal experiment design for identification, then it is uniformly discrete. 
\end{proposition}

In view of Proposition~\ref{prop:impossible_noisy}, to study universal experiment design for identification, we focus on sets of prior knowledge that are uniformly discrete. To that end, we first investigate conditions under which the data is $\Sigma_\text{pk}$--informative for identification. 

\begin{proposition}
\label{prop:data_best_noisy}
Let $\varepsilon\geq 0$, $\rho\geq 0$, $T\in\mathbb{N}$, and $\mathcal{D}=(u_{[0,T-1]},x_{[0,T]})\in\mathfrak{B}_T(A_\text{true},B_\text{true})$. Then, $\mathcal{D}$ is \mbox{$\Sigma_\text{pk}$--informative} for identification if
\begin{equation}
\label{eq:rank_data_generalized_noisy}
\norm{\begin{bmatrix}
F & G
\end{bmatrix}\begin{bmatrix}
\mathcal{H}_1(x_{[0,T-1]}) \\
\mathcal{H}_1(u_{[0,T-1]})
\end{bmatrix}}> 2\sqrt{T}\varepsilon,
\end{equation}
for all nonzero $(F,G)\in\Sigma_\text{pk}-\Sigma_\text{pk}$.
\end{proposition}

Proposition \ref{prop:data_best_noisy} is an extension of Corollary \ref{cor:inf_noise-free_general} to the noisy setting. The following example shows an application of Proposition~\ref{prop:data_best_noisy} to the identification of systems with integer parameters. 

\begin{example}
\label{ex:UED_noisy_unique}
Let $\varepsilon=0.01$ and consider the following prior knowledge set:
\begin{equation}
\Sigma_\text{pk}=\set{\left(\begin{bmatrix}
0 & 1  \\
a_{21} & a_{22} 
\end{bmatrix},\begin{bmatrix}
0 \\ 1
\end{bmatrix}\right)}{a_{12},a_{22}\in[-5,5]}.
\end{equation}
We recall that, in the notation of this paper, $[-5,5]$ is an integer interval. Hence, this prior knowledge set is uniformly discrete. Consider the true system to be \mbox{$(A_\text{true},B_\text{true})=\left(\begin{bmatrix}
0 & 1  \\
2 & -3 
\end{bmatrix},\begin{bmatrix}
0 \\ 1
\end{bmatrix}\right)$}. We apply the input $\mathcal{H}_1(u_{[0,2]})\!=\!\begin{bmatrix}
44 \!&\! 4 \!&\! 0
\end{bmatrix}$ and collect the state data
\begin{equation}
\mathcal{H}_1(x_{[0,3]})\!=\!\begin{bmatrix}
0 \!&\! 10 \!&\! 14 \!&\! -18 \\
10 \!&\! 14 \!&\! -18 \!&\! 82
\end{bmatrix}.
\end{equation}
By \cite[Fact~11.9.23(xii)]{bernstein2018scalar}, every nonzero \mbox{$(F,G)\in\Sigma_\text{pk}-\Sigma_\text{pk}$} satisfies $\norm{\begin{bmatrix}
F & G
\end{bmatrix}}\geq 1$. Using this inequality along with Lemma~\ref{lem:grcar(c)} in Appendix~\ref{app:I}, one can verify that for every nonzero \mbox{$(F,G)\in\Sigma_\text{pk}-\Sigma_\text{pk}$} we have
\begin{equation}
\norm{\begin{bmatrix}
F \! & \! G
\end{bmatrix}\!\!\begin{bmatrix}
\mathcal{H}_1(x_{[0,2]}) \\
\mathcal{H}_1(u_{[0,2]})
\end{bmatrix}}\!\geq\! \sigma_{n+m}\left(\!\begin{bmatrix}
\mathcal{H}_1(x_{[0,2]}) \\
\mathcal{H}_1(u_{[0,2]})
\end{bmatrix}\!\right)\!=\!15.2387.
\end{equation}
Thus, \eqref{eq:rank_data_generalized_noisy} holds for all $(F,G)\in\Sigma_\text{pk}-\Sigma_\text{pk}$ since \mbox{$2\sqrt{T}\varepsilon=0.0346$}. Hence, $\mathcal{D}$ is \mbox{$\Sigma_\text{pk}$--informative} for identification. This means that $\Sigma_\mathcal{D}\cap\Sigma_\text{pk}=\{(A_\text{true},B_\text{true})\}$.
\end{example}

Now, we turn our attention to universal experiment design for identification. The following theorem provides a sufficient condition under which an input is $\Sigma_\text{pk}$--universal for identification in the presence of noise. We recall the definition of $\Sigma^{(k)}$ from \eqref{eq:Sigma^{(k)}}. 

\begin{theorem}
\label{th:GM_noisy_generalized}
Suppose that $\Sigma_\text{pk}$ is uniformly discrete. Let $k\in[1,n]$ be such that $\Sigma_\text{pk}\subseteq\Sigma^{(k)}$. Then, $u_{[0,T-1]}$ is \mbox{$\Sigma_\textup{pk}$--universal} for identification if 
\begin{equation}
\label{eq:th:GM_noisy_generalized-1}
\begin{split}
&\norm{\begin{bmatrix}
F & G
\end{bmatrix}M_{A,B}^{(k)}D_A^{(k)}\mathcal{H}_{k+1}(u_{[0,T-1]})
}\\
&>\varepsilon\sqrt{T}\norm{d_A^{(k)}}_1\left(2+\norm{M_{A,I}^{(k)}}\norm{F}\sqrt{(n+1)}\right)
\end{split}
\end{equation}
for all $(A,B)\in\Sigma_\text{pk}$ and all nonzero $(F,G)\in\Sigma_\text{pk}-\Sigma_\text{pk}$.
\end{theorem}

Recall from Proposition~\ref{prop:impossible_noisy} that $\Sigma_\text{pk}$ can only enable universal experiment design for identification if it is uniformly discrete. The following theorem shows that, in certain situations, this condition is also sufficient. 

\begin{theorem}
\label{th:ctrl_nec_noisy_unique}
Let $\varepsilon>0$. Suppose that $\Sigma_\textup{pk}$ is bounded and satisfies $\gamma_\textup{pk}>0$. Then, $\Sigma_\textup{pk}$ enables universal experiment design for identification \emph{if and only if} it is uniformly discrete\footnote{If $\Sigma_\text{pk}$ is bounded and uniformly discrete, then it is finite.}.
\end{theorem}

We revisit Example~\ref{ex:UED_noisy_unique} to show that the input signal in that example is in fact $\Sigma_\text{pk}$--universal for identification. 

\begin{example}
Let $\varepsilon=0.01$ and consider the set $\Sigma_\text{pk}$ in Example~\ref{ex:UED_noisy_unique}. We show that the input signal \mbox{$\mathcal{H}_1(u_{[0,2]})=\begin{bmatrix}
44 & 4 & 0
\end{bmatrix}$} satisfies \eqref{eq:th:GM_noisy_generalized-1} for all $(A,B)\in\Sigma_\text{pk}$ and $(F,G)\in\Sigma_\text{pk}-\Sigma_\text{pk}$, and thus, it is \mbox{$\Sigma_\text{pk}$--universal} for identification. To this end, we observe that for every \mbox{$(A,B)\in\Sigma_\text{pk}$} we have
\begin{equation}
\begin{split}
d_A&=\begin{bmatrix}
\det A & -\tr A & 1\end{bmatrix}^\top=\begin{bmatrix}
-a_{21} & -a_{22} & 1\end{bmatrix}^\top, \\
M_{A,B}&=\begin{bmatrix}
0 & 0 & 1 \\
0 & 1 & a_{22} \\
1 & 0 & 0
\end{bmatrix},\text{ and } D_A=\begin{bmatrix}
-a_{21} & -a_{22} & 1 \\
-a_{22} & 1 & 0 \\
1 & 0 & 0
\end{bmatrix}
\end{split}
\end{equation}
for some $a_{21},a_{22}\in[-5,5]$. We also observe that every $(F,G)\in\Sigma_\text{pk}-\Sigma_\text{pk}$ is of the form $\begin{bmatrix}
F & G
\end{bmatrix}=\begin{bmatrix}
0 & 0 & 0  \\
f_1 & f_2 & 0 
\end{bmatrix}$
for some $f_1,f_2\in[-10,10]$. One can verify that the right-hand side of \eqref{eq:th:GM_noisy_generalized-1} admits the following upper bound:
\begin{equation}
\max_{\substack{(A,B)\in\Sigma_\text{pk}\\(F,G)\in\Sigma_\text{pk}-\Sigma_\text{pk}}}0.01\sqrt{2}\norm{d_A}_1\left(2+\sqrt{3}\norm{M_{A,I}}\norm{F}\right)<3.1.
\end{equation}
The left-hand side of \eqref{eq:th:GM_noisy_generalized-1} can be written as
\begin{equation}
\norm{\begin{bmatrix}
F & G
\end{bmatrix}M_{A,B}D_A\mathcal{H}_{k+1}(u_{[0,T-1]})
}=4\norm{11f_1+f_2},
\end{equation}
which is larger or equal to $4$ for all nonzero integers \mbox{$f_1,f_2\in[-10,10]$}. Therefore, \eqref{eq:th:GM_noisy_generalized-1} holds, and thus, $u_{[0,2]}$ is $\Sigma_\text{pk}$--universal for identification. 
\end{example}

%% file: i-o.tex
\section{Discussion on Input-Output Data}
\label{sec:i-o data}

In this section, we address universal experiment design for autoregressive systems with exogenous input (ARX systems) based on input-output data. This is achieved by reformulating ARX models in state-space form, thereby enabling the application of the results from Sections~\ref{sec:UED_noisefree} and~\ref{sec:i-o gain}, which were developed for state-space models with input-state data.

Let $M,L,p,m\in\mathbb{N}$. Consider the class of ARX systems
\begin{equation}
\label{eq:2}
y(t)=\sum_{i=1}^{L}P_iy(t-i)+\sum_{j=0}^{M}Q_ju(t-j)+v(t),
\end{equation}
where $y(t)\in\mathbb{R}^p$ is the output, $u(t)\in\mathbb{R}^m$ is the input, and $v(t)\in\mathbb{R}^p$ is the noise satisfying
\begin{equation}
\label{eq:ass2}
\|v(t)\|\leq\varepsilon\ \text{ for all }\ t\in\mathbb{Z}_+.
\end{equation}
We define \mbox{$P\coloneqq \begin{bmatrix}
P_L & P_{L-1} & \cdots & P_1
\end{bmatrix}\in\mathbb{R}^{p\times pL}$} and \mbox{$Q\coloneqq \begin{bmatrix}
Q_M & Q_{M-1} & \cdots & Q_0
\end{bmatrix}\in\mathbb{R}^{m\times m(M+1)}$}. We identify this class of systems with the set
\begin{equation}
\Sigma^\text{ARX}\coloneqq\mathbb{R}^{p\times pL}\times \mathbb{R}^{m\times m(M+1)},
\end{equation}
and we refer to the specific system \eqref{eq:2} as $(P,Q)\in\Sigma^\text{ARX}$. Consider the unknown true system $(P_\text{true},Q_\text{true})\in\Sigma^\text{ARX}$. Let $\Sigma^\text{ARX}_\text{pk}\subseteq\Sigma^\text{ARX}$ capture given prior knowledge of the true system, i.e.,
\begin{equation}
(P_\text{true},Q_\text{true})\in\Sigma^\text{ARX}_\text{pk}.
\end{equation}

We observe that \eqref{eq:2} can be written in the form of \eqref{eq:1}. To do so, we take $n=mM+pL$, 
\begin{equation}
\label{eq:state_and_noise}
x(t)=\begin{bmatrix}
u_{[t-M,t-1]} \\ 
y_{[t-L,t-1]}
\end{bmatrix},\text{ and } w(t)=\begin{bmatrix}
0 \\
I_p
\end{bmatrix}v(t).
\end{equation}
Furthermore, $A$ and $B$ are defined by
\begin{equation}
\label{eq:state-space-rep}
A=\begin{bmatrix}
A_{11} & A_{12} \\
\bar{Q} & P
\end{bmatrix}\ \text{ and }\ B=\begin{bmatrix}
B_1 \\ 
Q_0
\end{bmatrix},
\end{equation}
where \mbox{$\bar{Q}\coloneqq \begin{bmatrix}
Q_M & Q_{M-1} & \cdots & Q_1
\end{bmatrix}\in\mathbb{R}^{m\times mM}$}, 
\begin{equation}
\begin{split}
A_{11}&\!\!\coloneqq\!\! \begin{bmatrix}
0_{m(M-1)\times m} & I_{m(M-1)} \\
0_{m+p(L-1)\times m} & 0_{m+p(L-1)\times m(M-1)}
\end{bmatrix},\\ A_{12}&\!\!\coloneqq\!\!\begin{bmatrix}
0_{mM\times p} \!\!\!&\!\!\! 0_{mM\times p(L-1)} \\
0_{p(L-1) \times p} \!\!\!&\!\!\! I_{p(L-1)}
\end{bmatrix},\text{ and } B_1\!\!\coloneqq\!\!\begin{bmatrix}
0_{m(M-1)\times m} \\ I_m \\ 0_{p(L-1)\times m}
\end{bmatrix}\!\!.
\end{split}
\end{equation} 
The prior knowledge $\Sigma^\text{ARX}$ on the ARX system induces prior knowledge on the associated state-space model. This is captured by
\begin{equation}
\Sigma_\text{pk}=\set{\left(\begin{bmatrix}
A_{11} & A_{12} \\
\bar{Q} & P
\end{bmatrix},\begin{bmatrix}
B_1 \\ 
Q_0
\end{bmatrix}\right)}{(P,Q)\in\Sigma_\text{pk}^\text{ARX}}.
\end{equation}
This set $\Sigma_\text{pk}$ is not open (even if $\Sigma_\text{pk}^\text{ARX}$ is open in $\Sigma^\text{ARX}$). As such, in the noise-free case, the results in \mbox{Section~\ref{sec:noise-free_open}} cannot be directly applied, but one can use the results in Section~\ref{sec:noise-free_arbitrary} to find a \mbox{$\Sigma_\text{pk}$--universal} input for identification. In the presence of noise, we note that \eqref{eq:ass2} implies \eqref{eq:ass1}. Therefore, a \mbox{$\Sigma_\text{pk}$--universal} input for $\rho$--accuracy identification can be designed using the results presented in Section~\ref{sec:i-o gain}. 



%% file: examples.tex
\section{Examples}
\label{sec:ex}

\begin{table*}[h]
    \centering
    \caption{A summary of universal experiment design methods for $\rho$--accuracy identification and their feasibility}
    \label{table:summary}
    \begin{tabular}{|c|c|c|c|c|c|c|c|c|c|}
    \hline
     \multirow{2}{*}{Noise} & \multicolumn{4}{c|}{\multirow{2}{*}{Prior knowledge set}} & \multirow{2}{*}{Results} & \multicolumn{2}{c|}{Universal experiment design} & \multirow{2}{*}{$T\geq$} \\ \cline{7-8}
       & \multicolumn{4}{c|}{} & & $\rho=0$ & $\rho>0$ & \\
     \hline\hline
      \multirow{5}{*}{$\varepsilon=0$} & \multirow{3}{*}{Open} & \multicolumn{3}{c|}{\multirow{2}{*}{$\Sigma_\text{pk}\subseteq\Sigma_\text{cont}$}} & Willems' lemma & \multicolumn{2}{c|}{\multirow{2}{*}{feasible}} & $nm+n+m$ \\ \cline{6-6}\cline{9-9}
        &  & \multicolumn{3}{c|}{} & Thm. \ref{th:GM_noiseless} & \multicolumn{2}{c|}{} & $2n+m$ \\ \cline{3-9}
           &   & \multicolumn{3}{c|}{$\Sigma_\text{pk}\not\subseteq\Sigma_\text{cont}$} & Prop. \ref{prop:ctrl_nec} & \multicolumn{2}{c|}{infeasible} & NA \\ \cline{2-9}
           & \multirow{2}{*}{Arbitrary} & \multirow{2}{*}{$\Sigma_\text{pk}\subseteq\Sigma^{(k)}$} & \multicolumn{2}{c|}{$\Sigma_\text{pk}\subseteq\Sigma_\text{cont}$} & \multirow{2}{*}{Thm. \ref{th:GM_noiseless_generalized}} & \multicolumn{2}{c|}{feasible} & \multirow{2}{*}{$k+1$} \\ \cline{4-5} \cline{7-8}
           & & & \multicolumn{2}{c|}{$\Sigma_\text{pk}\not\subseteq\Sigma_\text{cont}$} &  & \multicolumn{2}{c|}{depends of $\Sigma_\text{pk}$} & \\ \hline\hline
           \multirow{4}{*}{$\varepsilon\neq 0$} & Not uniformly & \multirow{4}{*}{$\Sigma_\text{pk}\subseteq\Sigma^{(k)}$} & \multicolumn{2}{c|}{$\gamma_\text{pk}\geq0$} & Thm.~\ref{th:open_design_noisy} \& \ref{th:helping_in_exp} & \multirow{2}{*}{infeasible} & depends on $\Sigma_\text{pk}$ & \multirow{2}{*}{$k+n+m$} \\ \cline{5-6}\cline{8-8}
            &  discrete &  & \multicolumn{2}{c|}{Bounded \& $\gamma_\text{pk}>0$} & Thm. \ref{th:open_enabling_noisy} &  & feasible &  \\ 
            \cline{2-2}\cline{4-7}\cline{8-9}
           & Uniformly &  & \multicolumn{2}{c|}{$\gamma_\text{pk}\geq0$} & Thm. \ref{th:GM_noisy_generalized} & \multicolumn{2}{c|}{depends on $\Sigma_\text{pk}$} & \multirow{2}{*}{$k+1$} \\ \cline{4-8}
           & discrete & & \multicolumn{2}{c|}{Bounded \& $\gamma_\text{pk}>0$} & Thm.~\ref{th:ctrl_nec_noisy_unique} & \multicolumn{2}{c|}{feasible} & \\
           \hline
    \end{tabular}
\end{table*}

In this section, we consider two examples. The first example concerns finding a hands-off universal input in the noise-free setting. The second example considers a network of identical systems in the presence of noise. 

\subsection{Example 1: Relative orbital motion}
\label{subsec:ex1}

We consider the relative motion of two spacecraft, chaser and target, flying in adjacent circular orbits. This motion can be described by the so-called Clohessy-Wiltshire equations \cite[Sec. II]{sullivan2017comprehensive}. We discretized this dynamics using the forward Euler method, with sample time $\tau>0$. We assume that the output $y(t)\in\mathbb{R}^{2}$ is the position of the chaser relative to the target, and the input $u(t)\in\mathbb{R}^2$ is the force applied to the chaser. The true system belongs to the class of ARX models with $L=2$, $M=2$, $P_\text{true}=\mathcal{P}(\theta_\text{true})$, and $Q_\text{true}=\mathcal{Q}(m_\text{true})$,
where $\mathcal{P}:\mathbb{R}\rightarrow\mathbb{R}^{2\times 4}$ and $\mathcal{Q}:\mathbb{R}\rightarrow\mathbb{R}^{2\times 6}$ are defined as
\begin{equation}
\mathcal{P}(\theta)\!\coloneqq\!\begin{bmatrix}
3\tau^2\theta^2\!-\!1 \!\!&\!\! -2\tau\theta \!\!&\!\! 2 \!\!&\!\! 2\tau\theta \\
2\tau\theta \!\!&\!\! -1 \!\!&\!\! -2\tau\theta \!\!&\!\! 2
\end{bmatrix}
\end{equation}
and $\mathcal{Q}(m)\!\coloneqq\!\frac{\tau^2}{m}\begin{bmatrix}
I_2 & 0_{2\times 4}
\end{bmatrix}$, respectively. The mass of the chaser is denoted by \mbox{$m_\text{true}$}. Parameter \mbox{$\theta_\text{true}=\sqrt{\mu_\text{E}/r_\text{true}^3}$} is the orbital rate of the target, where $r_\text{true}$ is the true radius of the target's orbit and $\mu_\text{E}$ is the Earth's standard gravitational parameter. We assume that $r_\text{true}$ and $m_\text{true}$ are unknown. Hence, the parameter \mbox{$\theta_\text{true}$} is also unknown. We also assume that $\varepsilon=0$. The goal is to design a universal input leading to the identification of $\theta_\text{true}$ and $m_\text{true}$.

A state-space representation of \mbox{$(\mathcal{P}(\theta),\mathcal{Q}(m))\in\Sigma^\text{ARX}$} can be obtained by taking the state vector as \mbox{$x(t)=\begin{bmatrix}
y(t)^\top &
y(t+1)^\top
\end{bmatrix}^\top$} and the pair $(\mathcal{A}(\theta),\mathcal{B}(m))\in\Sigma$ as
\begin{equation}
\mathcal{A}(\theta)\!\coloneqq\!\begin{bmatrix}
0 \!\!&\!\! 0 \!\!&\!\! 1 \!\!&\!\! 0 \\
0 \!\!&\!\! 0 \!\!&\!\! 0 \!\!&\!\! 1 \\
3\tau^2\theta^2\!-\!1 \!\!&\!\! -2\tau\theta \!\!&\!\! 2 \!\!&\!\! 2\tau\theta \\
2\tau\theta \!\!&\!\! -1 \!\!&\!\! -2\tau\theta \!\!&\!\! 2
\end{bmatrix}\!,\ \mathcal{B}(m)\!\coloneqq\!\frac{\tau^2}{m}\begin{bmatrix}
0 \!\!&\!\! 0 \\
0 \!\!&\!\! 0 \\
1 \!\!&\!\! 0 \\
0 \!\!&\!\! 1 
\end{bmatrix}\!.
\end{equation}
We define
\begin{equation}
\label{eq:Sigma_pk_ex1}
\Sigma_\text{pk}=\set{(\mathcal{A}(\theta),\mathcal{B}(m))}{\theta,m> 0}.
\end{equation}
Since $\Sigma_\text{pk}\subset\Sigma_\text{cont}$, it follows from Proposition~\ref{prop:Willems} that, if $u_{[0,T-1]}$ is persistently exciting of order $5$ (which requires $T\geq14$), then it is \mbox{$\Sigma_\text{pk}$--universal} for identification. Nevertheless, due to practical concerns, a popular approach in spacecraft dynamics is to use hands-off input signals that take zero values over a significant amount of time, e.g., two-impulse inputs of the form \mbox{$u(0)\neq 0$}, \mbox{$u(T-1)\neq 0$}, and $u(t)=0$ for \mbox{$t\in[1,T-2]$} (see \cite{prussing1970optimal}). Such a two-impulse input signal is not persistently exciting of order $5$. Interestingly, it turns out that such input signals are typically \mbox{$\Sigma_\text{pk}$--universal} for identification. As an example, take $T=6$, \mbox{$u(0)=u(5)=\begin{bmatrix}
1 & 0
\end{bmatrix}^\top$}, and \mbox{$u_{[1,4]}=0$}. We use Theorem~\ref{th:GM_noiseless_generalized} with $k=n$. Observe that every \mbox{$(F,G)\in\Sigma_\text{pk}-\Sigma_\text{pk}$} satisfies 
\begin{equation}
F=\begin{bmatrix}
0 & 0 & 0 & 0 \\
0 & 0 & 0 & 0 \\
f_1 & -f_2 & 0 & f_2 \\
f_2 & 0 & -f_2 & 0
\end{bmatrix}\text{ and }G=\begin{bmatrix}
0 & 0 \\
0 & 0 \\
g & 0 \\
0 & g 
\end{bmatrix}
\end{equation}
for some $f_1,f_2,g\in\mathbb{R}$. Let $(A,B)\in\Sigma_\text{pk}$. We claim that if $\begin{bmatrix}
F & G
\end{bmatrix}M_{A,B}D_A\mathcal{H}_{5}(u_{[0,5]})= 0$, then both $F$ and $G$ are equal to zero. To see this, we verify that
\begin{equation}
\label{eq:ex1_nonzero}
\begin{bmatrix}
F \!\!&\!\! G
\end{bmatrix}M_{A,B}D_A\mathcal{H}_{5}(u_{[0,5]})\!=\!\begin{bmatrix}
0 \!&\! 0 \\
0 \!&\! 0 \\
-2\frac{\tau^3\theta}{m}f_2+(\tau^2\theta^2+ 1)g \!&\! g \\
\frac{\tau^2}{m}f_1-\frac{\tau^2}{m}f_2 \!&\! 0 
\end{bmatrix}\!.
\end{equation}
Now, suppose that $\begin{bmatrix}
F & G
\end{bmatrix}M_{A,B}D_A\mathcal{H}_{5}(u_{[0,T-1]})= 0$. We have $g=0$ due to the second column. Moreover, since $g=0$, from the $(3,1)$-th entry we have $f_2=0$. Next, we observe from the $(4,1)$-th entry that $f_1=0$. Now, it follows from Theorem~\ref{th:GM_noiseless_generalized} that $u_{[0,5]}$ is $\Sigma_\text{pk}$--universal for identification. In fact, using the same procedure, one can show that if $T\geq 6$, then \emph{almost any} two-impulse input is \mbox{$\Sigma_\text{pk}$--universal} for identification. We note that such two-impulse input signals may not yield $\Sigma$--informative data for identification (without using prior knowledge) as the generated data may not satisfy Proposition~\ref{prop:inf_noise-free_int(c)}. However, such an input yields data suitable for identification when prior knowledge is incorporated.

\subsection{Example 2: Network of unknown but identical systems}
\label{sub:network}

Consider the true system to be a network of $N=1000$ identical systems as
\begin{equation}
A_\text{true}=I_N\otimes \mathcal{A}(\alpha_\text{true})\ \text{ and }\ B_\text{true}=I_N\otimes \mathcal{B}(\beta_\text{true}),
\end{equation}
where $\mathcal{A}:\mathbb{R}^3\rightarrow\mathbb{R}^{3\times 3}$ and $\mathcal{B}:\mathbb{R}^2\rightarrow\mathbb{R}^{3\times 2}$ are defined as
\begin{equation}
\mathcal{A}(\alpha)\!\coloneqq\!\begin{bmatrix}
1 & 0.01 & 0 \\
0 & 1 & 0.01 \\
\alpha_1 & \alpha_2 & 1+\alpha_3
\end{bmatrix},\ \mathcal{B}(\beta)\!\coloneqq\!\begin{bmatrix}
0 & 0 \\
1 & \beta_1 \\
\beta_2 & 1
\end{bmatrix},
\end{equation}
with $\alpha_\text{true}=\begin{bmatrix}
0.01 & -0.2 & 0.04
\end{bmatrix}$ and $\beta_\text{true}=\begin{bmatrix}
-0.03 & 0.01
\end{bmatrix}$ that are unknown parameters. We assume that the noise $w(t)$ satisfies \eqref{eq:ass1} with $\varepsilon=10^{-4}$. Here, $n=3000$ and $m=2000$. We consider the prior knowledge set given by
\begin{equation}
\Sigma_\text{pk}=\set{(I_N\otimes \mathcal{A}(\alpha),I_N\otimes \mathcal{B}(\beta))}{\|\alpha\|,\|\beta\|\leq 0.1}.
\end{equation}
We aim at finding a $\Sigma_\text{pk}$--universal input for $\rho$--accuracy identification with $\rho=10^{-2}$. For this, we use Theorem~\ref{th:open_design_noisy}, i.e., we find an input $u_{[0,T-1]}$ such that \eqref{eq:th:open_design_noisy_main} holds for all $(A,B)\in\Sigma_\text{pk}$. 

For $(A,B)\in\Sigma_\text{pk}$, let $\alpha,\beta\in\mathbb{R}$ be such that $A=I_N\otimes \mathcal{A}(\alpha)$ and $B=I_N\otimes \mathcal{B}(\beta)$. We note that the degree of the minimal polynomial of $A$ is less than or equal to the degree of the characteristic polynomial of $\mathcal{A}(\alpha)$.
Therefore, $\deg(\mu_A)\leq 3$, and thus, we have $\Sigma_\text{pk}\subset\Sigma^{(3)}$. Moreover, one can verify that for all $i\in\mathbb{N}$ we have
\begin{equation}
\sigma_i(M_{A}^{(3)})\!=\!\sigma_i(M_{\mathcal{A}(\alpha)}^{(3)}) \text{ and } \sigma_{i}(M_{(A,B)}^{(3)})\!=\!\sigma_{i}(M_{(\mathcal{A}(\alpha),\mathcal{B}(\alpha))}^{(3)}).
\end{equation}
Now, to satisfy \eqref{eq:th:open_design_noisy_main} for all $(A,B)\in\Sigma_\text{pk}$, we take the input $u_{[0,T-1]}$ to be persistently exciting of order $4$ with a sufficiently large magnitude such that 
\begin{equation}
\label{eq:ineq_random_search}
\sigma_{8}(\mathcal{H}_{4}(u_{[0,T-1]}))\geq \varepsilon\sqrt{T}\tfrac{d_\text{pk}^{(3)}}{\gamma_\text{pk}}\left(\mu_\text{pk}^{(3)}\sqrt{k+1}+\tfrac{1}{\rho}\right),
\end{equation}
where $d_\text{pk}^{(3)}\coloneqq {\displaystyle\sup_{(A,B)\in\Sigma_\text{pk}}}\frac{\|d_A^{(3)}\|_1}{\sigma_{8}(D^{(k)}_A)}$, $\mu_\text{pk}^{(3)}\coloneqq {\displaystyle\sup_{(A,B)\in\Sigma_\text{pk}}}\|M_A^{(3)}\|$, and $\gamma_\text{pk}$ defined in \eqref{eq:gamma_def} are finite positive constants (see Appendix~\ref{app:I} for more details on why \eqref{eq:ineq_random_search} implies that \eqref{eq:th:open_design_noisy_main} holds for all $(A,B)\in\Sigma_\text{pk}$). We compute $d_\text{pk}^{(3)}$, $\mu_\text{pk}^{(3)}$, and $\gamma_\text{pk}$ numerically using the GlobalSearch algorithm in MATLAB, combined with the nonlinear programming solver \texttt{fmincon}, which yields the following values:
\begin{equation}
d_\text{pk}^{(3)}=128.8164,\ \mu_\text{pk}^{(3)}=1.9178,\text{ and } \gamma_\text{pk}=0.0141.
\end{equation}
We take $T=11$ and observe that \eqref{eq:ineq_random_search} reduces to
\begin{equation}
\label{eq:ineq_random_search-eval}
\sigma_{8}(\mathcal{H}_{4}(u_{[0,10]}))\geq 314.9665.
\end{equation}
Now, any persistently exciting input of order $4$, with a sufficiently large magnitude so that \eqref{eq:ineq_random_search-eval} holds, is $\Sigma_\text{pk}$--universal input for $\rho$--accuracy identification. We emphasize that the number of required samples for this experiment is $11$, while applying the robust version of Willems et al.'s fundamental lemma \cite{coulson2022quantitative} requires at least $nm+n+m=6005000$ data samples. 

\section{Conclusions}
\label{sec:conclusions}

In this paper, we have focused on offline experiment design for set-membership identification within the framework of universal inputs. We studied both exact and approximate identification in the noise-free and noisy settings. We investigated the role of prior knowledge in designing an experiment. We have shown that using suitable prior knowledge, one can design universal inputs that outperform persistently exciting ones and go beyond the framework of Willems et al.'s fundamental lemma. See Table~\ref{table:summary} for a summary of the results.

In this work, we only focused on experiment design for identification. Experiment design for data-driven control is an interesting topic that is left as future work. Moreover, due to safety reasons that arise in practice, experiment design with constraints on the system's trajectory is an important topic of research that is also left as future work. 

%% file: appendix.tex
\section{Appendix}

\renewcommand{\thesubsection}{\Alph{subsection}}

This section contains all the proofs of the main results. 





\subsection{Proof of Theorem~\ref{th:single-input-noiseless}}

We first present the following auxiliary result. 

\begin{lemma}[{\cite[prop. 3]{shakouri2025new}}]
\label{lem:contruction_m=1}
Suppose that $m=1$, $\varepsilon=0$, and $u_{[0,T-1]}$ is not persistently exciting of order $n+1$. Let $\eta_1,\ldots,\eta_n\in\mathbb{R}$, not all zero, be such that
\begin{equation}
\label{eq:eta_ker_m=1}
\begin{bmatrix}
\eta_0 & \cdots & \eta_n
\end{bmatrix}\mathcal{H}_{n+1}(u_{[0,T-1]})=0.
\end{equation}
Let $(A,B)\in\Sigma$ be such that $\sum_{i=0}^n \lambda^i \eta_i^\top\neq 0$ for all \mbox{$\lambda\in\spec A$}. Then, there exists $\mathcal{D}=(u_{[0,T-1]},x_{[0,T]})\in\mathfrak{B}(A,B,u_{[0,T-1]})$ such that $\rank \begin{bmatrix}
\mathcal{H}_1(x_{[0,T-1]}) \\ \mathcal{H}_1(u_{[0,T-1]})
\end{bmatrix}<n+m$. 
\end{lemma}\vspace{0.2 cm}

\textit{Proof of Theorem~\ref{th:single-input-noiseless}:} The ``if'' part follows from Proposition~\ref{prop:Willems}. For the ``only if'' part, assume that $u_{[0,T-1]}$ is not persistently exciting of order $n+1$. Let $\eta_0,\ldots,\eta_n\in\mathbb{R}$ satisfy \eqref{eq:eta_ker_m=1}. Since $\Sigma_\text{pk}$ is open, there exists $(A,B)\in\Sigma_\text{pk}$ such that $\sum_{i=0}^n \lambda^i \eta_i\neq 0$ for all $\lambda\in\spec A$. It follows now from Lemma \ref{lem:contruction_m=1} that there exists \mbox{$\mathcal{D}=(u_{[0,T-1]},x_{[0,T]})\in\mathfrak{B}(A,B,u_{[0,T-1]})$} satisfying $\rank \begin{bmatrix}
\mathcal{H}_1(x_{[0,T-1]})^\top & \mathcal{H}_1(u_{[0,T-1]})^\top
\end{bmatrix}<n+m$. Due to Proposition \ref{prop:inf_noise-free_int}, $\mathcal{D}$ is not \mbox{$\Sigma_\text{pk}$--informative} for identification. Therefore, the input $u_{[0,T-1]}$ is not \mbox{$\Sigma_\text{pk}$--universal} for identification. \hfill \QED

\subsection{Proof of Theorem~\ref{th:GM_noiseless}}

To prove Theorem~\ref{th:GM_noiseless}, we need some notation and an auxiliary result. Let $(A,B)\in\Sigma$ with $k\in[\deg(\mu_A),n]$, $p\in\mathbb{N}$, \mbox{$C\in\mathbb{R}^{p\times n}$}, and $D\in\mathbb{R}^{p\times m}$. We define 
\begin{equation}
\mathcal{M}(A,B,C,D)\coloneqq \begin{bmatrix}
D & 0 & \cdots & 0 \\
CB & D & \cdots & 0 \\
\vdots & \vdots & \ddots & \vdots \\
C A^{k-1} B & C A^{k-2} B & \cdots & D 
\end{bmatrix}.
\end{equation}
We also define $\Gamma\coloneqq \mathcal{M}(A,B,C,D)$, $\Theta\coloneqq\mathcal{M}(A,I,C,0)$, and $\Omega\coloneqq \begin{bmatrix}
C^\top & (CA)^\top & \cdots & (CA^{k})^\top
\end{bmatrix}^\top$. Recall the definitions of $M^{(k)}_{A,B}$, $d_A^{(k)}$, $D_A^{(k)}$, and $M^{(k)}_A$ from \eqref{eq:M_A,B^k}, \eqref{eq:d_A^k}, \eqref{eq:D_A^k}, and \eqref{eq:M_A^k}, respectively. Let $d_i$ denote the $i$th entry of $d_A^{(k)}$. We define
\begin{equation}
\begin{split}
\tilde{D}_{A}^{(k)}&\coloneqq \begin{bmatrix}
d_0 & \cdots & d_{k-1} & d_k \\
d_1 & \cdots & d_k & 0 \\
\vdots & \iddots & \vdots & \vdots \\
d_k & \cdots & 0 & 0
\end{bmatrix}\otimes I_n\ \text{ and }\\
\label{eq:def_Q}
Q_i&\coloneqq\begin{bmatrix}
0_{i\times (T-k)} \\ I_{T-k} \\ 0_{(k-i)\times(T-k)}
\end{bmatrix}\in\mathbb{R}^{T\times (T-k)}\ \text{ for } i\in[0,k].
\end{split}
\end{equation}
\begin{lemma}
\label{lem:Omega_Gamma_Theta}
Let $(A,B)\in\Sigma$ and $k\in[\deg(\mu_A),n]$. Consider the data $(u_{[0,T-1]},x_{[0,T]})\in\mathfrak{B}_T(A,B)$ and the noise signal $w_{[0,T-1]}$ satisfying \eqref{eq:1} for all $t\in[0,T-1]$. Let $p\in\mathbb{N}$, $C\in\mathbb{R}^{p\times n}$, $D\in\mathbb{R}^{p\times m}$. Define
\begin{equation}
\label{eq:output_y(t)}
y_{[0,T-1]}=C x_{[0,T-1]} + D u_{[0,T-1]}.
\end{equation}
Then, we have
\begin{equation}
\label{eq:main_identity}
\begin{split}
\mathcal{H}_1(y_{[0,T-1]})\sum_{i=0}^{k}d_iQ_i=\begin{bmatrix}
C & D
\end{bmatrix}M_{A,B}^{(k)}D_A^{(k)}\mathcal{H}_{k+1}(u_{[0,T-1]}) \\ 
+\begin{bmatrix}
C & 0
\end{bmatrix}M_{A}^{(k)}\tilde{D}_A^{(k)}\mathcal{H}_{k+1}(w_{[0,T-1]}).
\end{split}
\end{equation}
\end{lemma}
\begin{proof}
The Hankel matrix of depth $k+1$ can be related to that of depth $1$ by the following relation:
\begin{equation}
\label{eq:hankel_generalized}
\mathcal{H}_{k+1}(y_{[0,T-1]})\!=\!(I_{{k+1}}\!\otimes\! \mathcal{H}_1(y_{[0,T-1]}))\!\begin{bmatrix}
Q_0^\top \!\!&\!\! \cdots \!\!&\!\! Q_{k}^\top
\end{bmatrix}^\top\!\!\!. 
\end{equation}
Multiply \eqref{eq:hankel_generalized} from left by $(d_A^{(k)}\otimes I_{p})^\top$ to obtain the following expression for the left-hand side of \eqref{eq:main_identity}:
\begin{equation}
\label{eq:sys_rel_mult_LHS}
\mathcal{H}_1(y_{[0,T-1]})\sum_{i=0}^{k}d_iQ_i=(d_A^{(k)}\otimes I_{p})^\top\mathcal{H}_{k+1}(y_{[0,T-1]}).
\end{equation}
Now, it is enough to show that the right-hand side of \eqref{eq:sys_rel_mult_LHS} is equal to that of \eqref{eq:main_identity}. For this, we first note that $\mathcal{H}_{k+1}(y_{[0,T-1]})$ satisfies the following identity:
\begin{equation}
\label{eq:sys_rel}
\mathcal{H}_{k+1}(y_{[0,T-1]})=\begin{bmatrix}
\Omega & \Gamma & \Theta
\end{bmatrix}\begin{bmatrix}
\mathcal{H}_1(x_{[0,T-k-1]}) \\
\mathcal{H}_{k+1}(u_{[0,T-1]}) \\
\mathcal{H}_{k+1}(w_{[0,T-1]})
\end{bmatrix}.
\end{equation}
Since $\sum_{i=0}^k d_iA^i=0$, we have $(d_A^{(k)}\otimes I_{p})^\top\Omega=0$. Multiplying \eqref{eq:sys_rel} from left by $d_A^{(k)}\otimes I_{p}$ yields
\begin{equation}
\label{eq:sys_rel_mult}
\begin{split}
(d_A^{(k)}\otimes I_{p})^\top\mathcal{H}_{k+1}(y_{[0,T-1]})=(d_A^{(k)}\otimes I_{p})^\top\Gamma \mathcal{H}_{k+1}(u_{[0,T-1]})\\
+(d_A^{(k)}\otimes I_{p})^\top\Theta\mathcal{H}_{k+1}(w_{[0,T-1]}).
\end{split}
\end{equation}
Now, we substitute $(d_A^{(k)}\otimes I_{p})^\top\Gamma=\begin{bmatrix}
C & D
\end{bmatrix}M_{A,B}^{(k)}D_A^{(k)}$ and $(d_A^{(k)}\otimes I_{p})^\top\Theta=\begin{bmatrix}
C & 0
\end{bmatrix}M_{A}^{(k)}\tilde{D}_A^{(k)}$ into this identity to see that the right-hand sides of \eqref{eq:sys_rel_mult_LHS} and \eqref{eq:main_identity} are equal.
\end{proof}

\textit{Proof of Theorem~\ref{th:GM_noiseless}:} Suppose that $u_{[0,T-1]}$ satisfies \eqref{eq:th:GM_noiseless-1} for all \mbox{$(A,B)\in\Sigma_\text{pk}$}. We claim that $u_{[0,T-1]}$ is $\Sigma_\text{pk}$--universal for identification. To see this, let $(A,B)\in\Sigma_\text{pk}$ and \mbox{$\mathcal{D}=(u_{[0,T-1]},x_{[0,T]})\in\mathfrak{B}(A,B,u_{[0,T-1]})$}. It is enough to show that $\mathcal{D}$ is $\Sigma_\text{pk}$--informative for identification. Consider $y_{[0,T-1]}$ as in \eqref{eq:output_y(t)} with $C=\begin{bmatrix}
I_n & 0
\end{bmatrix}^\top$ and $D=\begin{bmatrix}
0 & I_m
\end{bmatrix}^\top$. Based on Lemma~\ref{lem:Omega_Gamma_Theta}, since $w_{[0,T-1]}=0$ and $k=n$, we have
\begin{equation}
\label{eq:GM_formula}
\begin{bmatrix}
\mathcal{H}_{1}(x_{[0,T-1]}) \\ \mathcal{H}_{1}(u_{[0,T-1]})
\end{bmatrix}\sum_{i=0}^{n}d_iQ_i=M_{A,B}D_A\mathcal{H}_{n+1}(u_{[0,T-1]}).
\end{equation}
Since \eqref{eq:th:GM_noiseless-1} holds, we have $\rank \begin{bmatrix}
\mathcal{H}_{1}(x_{[0,T-1]}) \\ \mathcal{H}_{1}(u_{[0,T-1]})
\end{bmatrix}=n+m$. By Proposition \ref{prop:inf_noise-free_int}, this implies that $\mathcal{D}$ is \mbox{$\Sigma_\text{pk}$--informative} for identification. \hfill \QED

\subsection{Proof of Theorem~\ref{th:PE_nec_ball}}

First, we recall the following lemma from \cite{shakouri2025new}. 

\begin{lemma}[{\cite[Lem. 5]{shakouri2025new}}]
\label{lem:contruction}
Suppose that $\varepsilon=0$ and $u_{[0,T-1]}$ is not persistently exciting of order $n+1$. Let $\eta_0,\ldots,\eta_n\in\mathbb{R}^m$, not all zero, be such that
\begin{equation}
\label{eq:eta_ker}
\begin{bmatrix}
\eta_0^\top & \cdots & \eta_n^\top
\end{bmatrix}\mathcal{H}_{n+1}(u_{[0,T-1]})=0.
\end{equation}
Let $A\in\mathbb{R}^{n\times n}$ and $\zeta\in\mathbb{R}^n$ be such that $(A,\zeta)$ is controllable and $\sum_{i=0}^n \lambda^i \eta_i^\top\neq 0$ for all $\lambda\in\spec A$. Let \linebreak $B=\sum_{i=0}^n A^i\zeta\eta_i^\top$. Then, the pair $(A,B)$ is controllable and there exists $x_{[0,T]}$ satisfying $(u_{[0,T-1]},x_{[0,T]})\in\mathfrak{B}_T(A,B)$ such that $\rank \mathcal{H}_1(x_{[0,T-1]})<n$. 
\end{lemma}

\textit{Proof of Theorem~\ref{th:PE_nec_ball}:} The ``if'' part follows from Proposition \ref{prop:Willems}. To prove the ``only if'' part, suppose that $u_{[0,T-1]}$ is not persistently exciting of order $n+1$. Let \mbox{$\eta_0,\ldots,\eta_n\in\mathbb{R}^m$}, not all zero, satisfy \eqref{eq:eta_ker} and $\|\eta_i\|\leq 1$ for all \mbox{$i\in[0,n]$}. We claim that there exists \mbox{$(A,B)\in\Sigma_\text{pk}$} and $x_{[0,T]}$ satisfying \mbox{$(u_{[0,T-1]},x_{[0,T]})\in\mathfrak{B}_T(A,B)$} such that the condition in Proposition~\ref{prop:inf_noise-free_int(c)} does not hold. To see this, take \mbox{$A\in\mathcal{A}_\text{pk}$} and \mbox{$\zeta\in\mathbb{R}^n$} such that the pair $(A,\zeta)$ is controllable, \mbox{$\|\zeta\|< \beta/(\sum_{i=0}^n\|A^i\|)$}, and \mbox{$\sum_{i=0}^n \lambda^i \eta_i^\top\neq 0$} for all $\lambda\in\spec A$. Define \mbox{$B=\sum_{i=0}^n A^i\zeta\eta_i$}. Since \mbox{$\|B\|\leq\sum_{i=0}^n \|A^i\|\|\zeta\|\|\eta_i\|<\beta$}, we have \mbox{$(A,B)\in\Sigma_\text{pk}$}. Now, it follows from Lemma \ref{lem:contruction} that there exists $x_{[0,T]}$ satisfying $(u_{[0,T-1]},x_{[0,T]})\in\mathfrak{B}_T(A,B)$ such that the Proposition~\ref{prop:inf_noise-free_int(c)} does not hold. \hfill \QED

\subsection{Proof of Proposition~\ref{prop:data_best}}

For the ``if'' part, suppose that \eqref{eq:rank_data_gen} holds. Let \mbox{$(A,B)\in\Sigma_\mathcal{D}\cap\Sigma_\text{pk}$}. Since \mbox{$(\hat{A},\hat{B}),(A,B)\in\Sigma_\mathcal{D}$}, we have
\begin{equation}
\label{eq:difference_in_ker}
\begin{bmatrix}
\hat{A}-A & \hat{B}-B
\end{bmatrix}\begin{bmatrix}
\mathcal{H}_1(x_{[0,T-1]}) \\
\mathcal{H}_1(u_{[0,T-1]})
\end{bmatrix}=0.
\end{equation}
Since $(\hat{A},\hat{B}),(A,B)\in\Sigma_\text{pk}$, condition \eqref{eq:rank_data_gen} implies that $A=\hat{A}$ and $B=\hat{B}$. Therefore, $\Sigma_\mathcal{D}\cap\Sigma_\text{pk}$ is a singleton. We prove the ``only if'' part by contraposition. Suppose that \eqref{eq:rank_data_gen} does not hold. Take $(A,B)\in\Sigma_\text{pk}\backslash\{(\hat{A},\hat{B})\}$ such that \eqref{eq:difference_in_ker} holds. This implies that
\begin{equation}
\begin{split}
&A\mathcal{H}_1(x_{[0,T-1]})+B\mathcal{H}_1(u_{[0,T-1]}) \\
&=\hat{A}\mathcal{H}_1(x_{[0,T-1]})+\hat{B}\mathcal{H}_1(u_{[0,T-1]})=\mathcal{H}_1(x_{[1,T]}).
\end{split}
\end{equation}
Thus, $(u_{[0,T-1]},x_{[0,T]})\in\mathfrak{B}_T(A,B)$, which implies that $(A,B)\in\Sigma_\mathcal{D}$. Since  $(\hat{A},\hat{B}),(A,B)\in\Sigma_\mathcal{D}\cap\Sigma_\text{pk}$, the set \linebreak $\Sigma_\mathcal{D}\cap\Sigma_\text{pk}$ is not a singleton. Therefore, $\mathcal{D}$ is not \mbox{$\Sigma_\text{pk}$--informative} for identification. \hfill \QED


\subsection{Proof of Theorem~\ref{th:GM_noiseless_generalized}}

Suppose that $u_{[0,T-1]}$ satisfies \eqref{eq:th:GM_noiseless_generalized-1} for all \mbox{$(F,G)\in\Sigma_\text{pk}-\Sigma_\text{pk}$} and all \mbox{$(A,B)\in\Sigma_\text{pk}$}. To show that $u_{[0,T-1]}$ is $\Sigma_\text{pk}$--universal for identification, let $(A,B)\in\Sigma_\text{pk}$ and $\mathcal{D}=(u_{[0,T-1]},x_{[0,T]})\in\mathfrak{B}(A,B,u_{[0,T-1]})$. Also, let \mbox{$(F,G)\in\Sigma_\text{pk}-\Sigma_\text{pk}$}. Define $y_{[0,T-1]}$ as in \eqref{eq:output_y(t)} with $C=F$ and $D=G$. In view of Lemma~\ref{lem:Omega_Gamma_Theta}, since $w_{[0,T-1]}=0$, 
\begin{equation}
\label{eq:th:GM_noiseless_generalized-6}
\mathcal{H}_1(y_{[0,T-1]})\sum_{i=0}^{k}d_iQ_i=\begin{bmatrix}
F & G
\end{bmatrix}M_{A,B}^{(k)}D_A^{(k)}\mathcal{H}_{k+1}(u_{[0,T-1]}).
\end{equation}
The right-hand side of this equation is nonzero due to \eqref{eq:th:GM_noiseless_generalized-1}. Thus, we have $\mathcal{H}_1(y_{[0,T-1]})=\begin{bmatrix}
F & G
\end{bmatrix}\begin{bmatrix}
\mathcal{H}_{1}(x_{[0,T-1]}) \\ \mathcal{H}_{1}(u_{[0,T-1]})
\end{bmatrix}\neq 0$. Since this holds for all \mbox{$(F,G)\in\Sigma_\text{pk}-\Sigma_\text{pk}$}, it follows from Corollary~\ref{cor:inf_noise-free_general} that $\mathcal{D}$ is \mbox{$\Sigma_\text{pk}$--informative} for identification. Now, since this argument holds for all $(A,B)\in\Sigma_\text{pk}$, the input $u_{[0,T-1]}$ is \mbox{$\Sigma_\text{pk}$--universal} for identification. \hfill \QED

\subsection{Proof of Lemma~\ref{lem:noisy_rad_lower}}
\label{app:F}

To prove Lemma~\ref{lem:noisy_rad_lower}, we need some notation and auxiliary results. The set of $n\times n$ symmetric matrices is denoted by $\mathbb{S}^n$. For a matrix $N\in\mathbb{S}^{q+p}$, we consider the block-partitioned form
\begin{equation}
\label{eq:N_partition}
N=\begin{bmatrix}
N_{11} & N_{12} \\
N_{21} & N_{22}
\end{bmatrix},
\end{equation}
with $N_{11}\in\mathbb{S}^q$, $N_{22}\in\mathbb{S}^p$, and $N_{12}=N_{21}^\top\in\mathbb{R}^{q\times p}$. We denote the (generalized) Schur complement of $N$ with respect to $N_{22}$ by $N|N_{22}=N_{11}-N_{12}N_{22}^\dagger N_{21}$. Moreover, we denote
\begin{equation}
\mathcal{Z}_p(N)\coloneqq\set{Z\in\mathbb{R}^{p\times q}}{\begin{bmatrix}
I_q & Z^\top
\end{bmatrix} N \begin{bmatrix}
I_q & Z^\top
\end{bmatrix}^\top\geq 0}.
\end{equation}


\begin{lemma}
\label{lem:exp_qmi}
Let $(\hat{A},\hat{B})\in\Sigma$, $T\in\mathbb{N}$. Consider the data \mbox{$\mathcal{D}=(u_{[0,T-1]},x_{[0,T]})\in\mathfrak{B}_T(\hat{A},\hat{B})$}. Define $N_\text{in}$ and $N_\text{out}$ as follows:
\begin{subequations}
\begin{align}
\label{eq:N_in}
N_\text{in}&\coloneqq\begin{bmatrix}
\varepsilon^2 I & 0 & 0 \\
0 & 0 & 0 \\
0 & 0 & 0
\end{bmatrix}\!-\! \sum_{t=0}^{T-1}\begin{bmatrix}
x(t+1) \\
-x(t) \\
-u(t)
\end{bmatrix}\!\!\!\begin{bmatrix}
x(t+1) \\
-x(t) \\
-u(t)
\end{bmatrix}^\top\!\!, \\
\label{eq:N_out}
N_\text{out}&\coloneqq\begin{bmatrix}
T \varepsilon^2 I & 0 & 0 \\
0 & 0 & 0 \\
0 & 0 & 0
\end{bmatrix}-\sum_{t=0}^{T-1}\begin{bmatrix}
x(t+1) \\
-x(t) \\
-u(t)
\end{bmatrix}\begin{bmatrix}
x(t+1) \\
-x(t) \\
-u(t)
\end{bmatrix}^\top\!\!.
\end{align}
\end{subequations}
Then, the following set inclusions hold:
\begin{subequations}
\begin{align}
\label{eq:lem:exp_qmi-2}
\Sigma_\mathcal{D}&\supseteq\set{(A,B)\in\Sigma}{\begin{bmatrix}
A & B
\end{bmatrix}^\top\in\mathcal{Z}_{n+m}(N_\text{in})},\\
\label{eq:lem:exp_qmi-1}
\Sigma_\mathcal{D}&\subseteq \set{(A,B)\in\Sigma}{\begin{bmatrix}
A & B
\end{bmatrix}^\top\in\mathcal{Z}_{n+m}(N_\text{out})}.
\end{align}
\end{subequations}
\end{lemma}\vspace{0.25 cm}
\begin{proof}
To prove \eqref{eq:lem:exp_qmi-2}, let $\begin{bmatrix}
A & B
\end{bmatrix}^\top\in\mathcal{Z}_{n+m}(N_\textup{in})$. We claim that $(A,B)\in\Sigma_\mathcal{D}$. To show this, we observe that $(A,B)$ satisfies 
\begin{equation}
\label{eq:lem:exp_qmi_pf-4}
\begin{bmatrix}
I \\ A^\top \\ B^\top
\end{bmatrix}^{\!\top}\!\!\begin{bmatrix}
\mathcal{H}_1(x_{[1,T]}) \\
-\mathcal{H}_1(x_{[0,T-1]}) \\
-\mathcal{H}_1(u_{[0,T-1]})
\end{bmatrix}\!\!\begin{bmatrix}
\mathcal{H}_1(x_{[1,T]}) \\
-\mathcal{H}_1(x_{[0,T-1]}) \\
-\mathcal{H}_1(u_{[0,T-1]})
\end{bmatrix}^{\!\top}\!\!\begin{bmatrix}
I \\ A^\top \\ B^\top
\end{bmatrix} \leq  \varepsilon^2 I_n.
\end{equation}
This implies that
\begin{equation}
\label{eq:lowerbound_pf-1}
\|\mathcal{H}_1(x_{[1,T]})-A\mathcal{H}_1(x_{[0,T-1]})-B\mathcal{H}_1(u_{[0,T-1]})\|\leq \varepsilon.
\end{equation} 
Now, take the noise signal $w_{[0,T-1]}$ such that
\begin{equation}
\label{eq:lowerbound_pf-2}
\mathcal{H}_1(w_{[0,T-1]})\!=\!\mathcal{H}_1(x_{[1,T]})-A\mathcal{H}_1(x_{[0,T-1]})-B\mathcal{H}_1(u_{[0,T-1]}).
\end{equation}
According to \eqref{eq:lowerbound_pf-1}, this noise signal satisfies $\|\mathcal{H}_1(w_{[0,T-1]})\|\leq\varepsilon$. Therefore, we have $\|w(t)\|\leq\varepsilon$ for all $t\in[0,T-1]$. Hence, $w_{[0,T-1]}$ satisfies the noise model \eqref{eq:ass1}, and thus, $(A,B)\in\Sigma_\mathcal{D}$. 

To prove \eqref{eq:lem:exp_qmi-1}, let $(A,B)\in\Sigma_\mathcal{D}$. We claim that $\begin{bmatrix}
A & B
\end{bmatrix}^\top\in\mathcal{Z}_{n+m}(N_\textup{out})$. To see this, we observe that for any $t\in[0,T-1]$ the pair $(A,B)$ satisfies
\begin{equation}
\label{eq:lem:exp_qmi_pf-1}
\begin{bmatrix}
I & A & B
\end{bmatrix}\begin{bmatrix}
x(t+1)^\top &
-x(t)^\top &
-u(t)^\top
\end{bmatrix}^\top=-w(t),
\end{equation}
for some $w(t)$ with $\norm{w(t)}\leq\varepsilon$. Since we have \mbox{$w(t)w(t)^\top \leq \varepsilon^2 I_n$}, for every $t\in[0,T-1]$ the pair $(A,B)$ satisfies
\begin{equation}
\label{eq:lem:exp_qmi_pf-3}
\begin{bmatrix}
I & A & B
\end{bmatrix} N_t\begin{bmatrix}
I & A & B
\end{bmatrix}^\top\geq 0,
\end{equation}
where 
\begin{equation}
N_t\coloneqq\begin{bmatrix}
\varepsilon^2 I & 0 & 0 \\
0 & 0 & 0 \\
0 & 0 & 0
\end{bmatrix}-\begin{bmatrix}
x(t+1) \\
-x(t) \\
-u(t)
\end{bmatrix}\begin{bmatrix}
x(t+1) \\
-x(t) \\
-u(t)
\end{bmatrix}^\top.
\end{equation}
Now, taking the sum of the left-hand side of \eqref{eq:lem:exp_qmi_pf-3} over all \mbox{$t\in[0,T-1]$} yields \mbox{$\begin{bmatrix}
I \!&\! A \!&\! B
\end{bmatrix}(\textstyle{\sum_{t=0}^{T-1}N_t})\begin{bmatrix}
I \!&\! A \!&\! B
\end{bmatrix}^\top\!\geq\! 0$}. Since $N_\text{out}\!=\!\sum_{t=0}^{T-1} N_t$, we have \mbox{$\begin{bmatrix}
A \!&\! B
\end{bmatrix}^\top\!\in\!\mathcal{Z}_{n+m}(N_\text{out})$}. 
\end{proof}

\begin{lemma}[{\cite[Thm. 7]{shakouri2025chebyshev}}]
\label{lem:radcent_qmi}
Let $N\in\mathbb{R}^{q+p}$, partitioned as \eqref{eq:N_partition}, be such that $N_{22}< 0$ and $N|N_{22}\geq 0$. Then, we have $-N_{22}^{-1}N_{12}^\top\in\cent \mathcal{Z}_p(N)$ and $\rad \mathcal{Z}_p(N) = \sqrt{\sigma_1(N|N_{22})/\sigma_p(N_{22})}$
\end{lemma}

\textit{Proof of Lemma~\ref{lem:noisy_rad_lower}:}
(a) Let $(A,B)\in\Sigma$ and \mbox{$\mathcal{D}=(u_{[0,T-1]},x_{[0,T]})\in\mathfrak{B}_T(A,B)$}. Suppose that $\sigma_\mathcal{D}=0$. It follows from \cite[Thm. 3.2(b)]{HenkQMI2023} that $\mathcal{Z}_{n+m}(N_\text{in})$ is unbounded. Therefore, due to Lemma~\ref{lem:exp_qmi}, we have that $\Sigma_\mathcal{D}$ is unbounded. Now, suppose that $\sigma_\mathcal{D}>0$. It from \cite[Thm. 3.2(b)]{HenkQMI2023} that $\mathcal{Z}_{n+m}(N_\text{out})$ is bounded. Therefore, by \eqref{eq:lem:exp_qmi-1}, $\Sigma_\mathcal{D}$ is bounded and
\begin{equation}
\label{eq:th:noisy_radcompare-pf1}
\rad\Sigma_\mathcal{D}\leq\rad\mathcal{Z}_{n+m}(N_\text{out}).
\end{equation}
Now, partition $N_\text{out}$ as $N_\text{out}=\begin{bmatrix}
N_{11} & N_{12} \\
N_{21} & N_{22}
\end{bmatrix}$, where \mbox{$N_{11}\in\mathbb{S}^n$}, $N_{22}\in\mathbb{S}^{n+m}$, and $N_{12}=N_{12}^\top\in\mathbb{R}^{n\times (n+m)}$. We have $N_{22}<0$ because $\sigma_\mathcal{D}>0$. Since $\mathcal{Z}_{n+m}(N_\text{out})$ is nonempty, we have \mbox{$N|N_{22}\geq 0$}, see \cite[p. 2257]{HenkQMI2023}. Hence, $N_\text{out}$ satisfies the hypothesis of Lemma~\ref{lem:radcent_qmi}. Thus, we have
\begin{equation}
\label{eq:th:noisy_radcompare-pf2}
\rad \mathcal{Z}_{n+m}(N_\text{out})=\sqrt{\sigma_1(N|N_{22})/\sigma_{n+m}(N_{22})}.
\end{equation}
Since $N|N_{22}\leq\varepsilon^2 T I_n$, we have $\sigma_1(N|N_{22})\leq \varepsilon^2 T$. This, together with \eqref{eq:th:noisy_radcompare-pf1} and \eqref{eq:th:noisy_radcompare-pf2}, implies that $\rad\Sigma_\mathcal{D}\leq \frac{\sqrt{T}\varepsilon}{\sigma_\mathcal{D}}$. 

(b) Take $\mathcal{D}=(u_{[0,T-1]},x_{[0,T]})\in\mathfrak{B}(A,B,u_{[0,T-1]})$ to be a dataset corresponding to the noise signal $w_{[0,T-1]}=0$, i.e., $\mathcal{D}$ satisfies
\begin{equation}
\label{eq:data_ad_lower}
\mathcal{H}_1(x_{[1,T]})=A\mathcal{H}_1(x_{[0,T-1]})+B\mathcal{H}_1(u_{[0,T-1]}).
\end{equation}
Suppose that $\sigma_\mathcal{D}>0$. We will show that $\rad\Sigma_\mathcal{D}\geq \frac{\varepsilon}{\sigma_\mathcal{D}}$. Based on Lemma \ref{lem:exp_qmi}, we have
\begin{equation}
\label{eq:th:noisy_radcompare-pf1}
\rad\Sigma_\mathcal{D}\geq\rad\mathcal{Z}_{n+m}(N_\text{in}).
\end{equation}
Partition $N_\text{in}$ as $N_\text{in}=\begin{bmatrix}
N_{11} & N_{12} \\
N_{21} & N_{22}
\end{bmatrix}$, where $N_{11}\in\mathbb{S}^n$, $N_{22}\in\mathbb{S}^{n+m}$, and $N_{12}=N_{21}^\top\in\mathbb{R}^{n\times (n+m)}$. Since $\sigma_\mathcal{D}>0$, we have $N_{22}<0$. Since $\begin{bmatrix}
A & B
\end{bmatrix}^\top\in\mathcal{Z}_{n+m}(N_\text{in})$, we have that $\mathcal{Z}_{n+m}(N_\text{in})$ is nonempty. This, together with $N_{22}<0$, implies that \mbox{$N|N_{22}\geq 0$}, see \cite[p. 2257]{HenkQMI2023}. Now, it follows from Lemma~\ref{lem:radcent_qmi} that
\begin{equation}
\label{eq:th:noisy_radcompare-pf2}
\rad \mathcal{Z}_{n+m}(N_\text{in})=\sqrt{\sigma_1(N|N_{22})/\sigma_{n+m}(N_{22})}.
\end{equation}
We observe that $N|N_{22}\leq\varepsilon^2 I$. It follows from \cite[Eq. (3.4)]{HenkQMI2023} that $N|N_{22}\geq \begin{bmatrix}
I \!&\! A \!&\! B
\end{bmatrix} N_\text{in} \begin{bmatrix}
I \!&\! A \!&\! B
\end{bmatrix}^\top$. This, together with \eqref{eq:data_ad_lower}, implies $N|N_{22}\geq \varepsilon^2 I$. Therefore, \mbox{$N|N_{22}=\varepsilon^2 I$}. Substituting this into \eqref{eq:th:noisy_radcompare-pf2} yields $\rad \mathcal{Z}_{n+m}(N_\text{in})=\frac{\varepsilon}{\sigma_\mathcal{D}}$. Therefore, \eqref{eq:th:noisy_radcompare-pf1} implies \mbox{$\rad\Sigma_\mathcal{D}\geq \frac{\varepsilon}{\sigma_\mathcal{D}}$}. 
\hfill \QED

\subsection{Proof of Lemma~\ref{lem:open_design_noisy}}

Before stating the proof, we recall some singular value inequalities in the following lemma.

\begin{lemma}[{\cite[Thm. 3.3.16]{horn1994topics}}]
\label{lem:horn}
Let $M,N\in\mathbb{R}^{n\times m}$ and define $q=\min\{n,m\}$. The following inequalities hold:
\begin{enumerate}[label=(\alph*),ref=\ref{lem:horn}(\alph*)]
    \item\label{lem:horn(c)} $\sigma_{q}(M)-\sigma_1(N)\leq\sigma_q(M+N)$,
    \item\label{lem:horn(d)} $\sigma_q(NM^\top)\leq \sigma_q(N)\sigma_1(M)$.
\end{enumerate}
\end{lemma}

\textit{Proof of Lemma~\ref{lem:open_design_noisy}:} Suppose that $u_{[0,T-1]}$ satisfies~\eqref{eq:th:open_design_noisy_main} for some $\rho>0$. Let \mbox{$\mathcal{D}=(u_{[0,T-1]},x_{[0,T]})\in\mathfrak{B}(A,B,u_{[0,T-1]})$}. We show that \mbox{$\rad\Sigma_\mathcal{D}\leq\rho$}. Let $w_{[0,T-1]}$ be the noise signal that satisfies $\|w(t)\|\leq\varepsilon$ for all $t\in[0,T-1]$ and
\begin{equation}
\mathcal{H}_1(x_{[1,T]})\!=\!A\mathcal{H}_1(x_{[0,T-1]})+B\mathcal{H}_1(u_{[0,T-1]})+\mathcal{H}_1(w_{[0,T-1]}).
\end{equation}
Define $y_{[0,T-1]}$ as in \eqref{eq:output_y(t)} with $C=\begin{bmatrix}
I_n & 0
\end{bmatrix}^{\top}$ and \mbox{$D=\begin{bmatrix}
0 & I_m
\end{bmatrix}^{\top}$}. It follows from Lemma~\ref{lem:Omega_Gamma_Theta} that
\begin{equation}
\label{eq:lem:open_design_noisy_pf-4}
\begin{split}
\begin{bmatrix}
\mathcal{H}_{1}(x_{[0,T-1]}) \\ \mathcal{H}_{1}(u_{[0,T-1]})
\end{bmatrix}\sum_{i=0}^{k}d_iQ_i=M_{A,B}^{(k)}D_A^{(k)}\mathcal{H}_{k+1}(u_{[0,T-1]})\\
+\begin{bmatrix}
I_n & 0 \\
0 & 0
\end{bmatrix}M_{A}^{(k)}\tilde{D}_A^{(k)}\mathcal{H}_{k+1}(w_{[0,T-1]}).
\end{split}
\end{equation}
For the right-hand side, it follows from Lemma~\ref{lem:horn(c)} that
\begin{equation}
\label{eq:lem:open_design_noisy_pf-4_RHS}
\begin{split}
&\sigma_{n+m}\bigg(M_{A,B}^{(k)}D_A^{(k)}\mathcal{H}_{k+1}(u_{[0,T-1]})\\
&+\begin{bmatrix}
I_n & 0 \\
0 & 0
\end{bmatrix}M_{A}^{(k)}\tilde{D}_A^{(k)}\mathcal{H}_{k+1}(w_{[0,T-1]})\bigg)\\
&\geq \sigma_{n+m}(M_{A,B}^{(k)}D_A^{(k)}\mathcal{H}_{k+1}(u_{[0,T-1]}))\\
&-\norm{M_{A}^{(k)}\tilde{D}_A^{(k)}\mathcal{H}_{k+1}(w_{[0,T-1]})}.
\end{split}
\end{equation}
Based on Lemma~\ref{lem:horn(d)}, the smallest singular value of the left-hand side of \eqref{eq:lem:open_design_noisy_pf-4} satisfies
\begin{equation}
\label{eq:lem:open_design_noisy_pf-4_LHS}
\sigma_{n+m}\left(\begin{bmatrix}
\mathcal{H}_{1}(x_{[0,T-1]}) \\ \mathcal{H}_{1}(u_{[0,T-1]})
\end{bmatrix}\sum_{i=0}^{k}d_iQ_i\right)\leq \sigma_\mathcal{D}\bigg\|\sum_{i=0}^{k}d_iQ_i\bigg\|.
\end{equation}
Now, \eqref{eq:lem:open_design_noisy_pf-4}, \eqref{eq:lem:open_design_noisy_pf-4_RHS}, and \eqref{eq:lem:open_design_noisy_pf-4_LHS} imply that
\begin{equation}
\label{eq:th:ctrl_nec_noisy-2}
\begin{split}
\sigma_\mathcal{D}\bigg\|\sum_{i=0}^{k}d_iQ_i\bigg\|\geq\sigma_{n+m}(M_{A,B}^{(k)}D_A^{(k)}\mathcal{H}_{k+1}(u_{[0,T-1]}))\\
-\norm{M_{A}^{(k)}\tilde{D}_A^{(k)}\mathcal{H}_{k+1}(w_{[0,T-1]})}.
\end{split}
\end{equation}
Observe that, since $\|Q_i\|=1$ for all $i\in[0,n]$, we have $\|\sum_{i=0}^{k}d_iQ_i\|\leq \sum_{i=0}^{k}|d_i|=\|d_A^{(k)}\|_1$. Also, it follows from the triangle inequality that $\|\tilde{D}_A^{(k)}\|\leq\|d_A^{(k)}\|_1$. Hence, \eqref{eq:th:ctrl_nec_noisy-2} yields
\begin{equation}
\label{eq:th:ctrl_nec_noisy-3}
\begin{split}
\sigma_\mathcal{D}\geq\tfrac{1}{\|d_A^{(k)}\|_1}\sigma_{n+m}(M_{A,B}^{(k)}D_A^{(k)}\mathcal{H}_{k+1}(u_{[0,T-1]})) \\
-\|M_{A}^{(k)}\|\norm{\mathcal{H}_{k+1}(w_{[0,T-1]})}.
\end{split}
\end{equation}
Now, we use the noise model~\eqref{eq:ass1} to find an upper bound for $\norm{\mathcal{H}_{k+1}(w_{[0,T-1]})}$ in terms of $\varepsilon$. Based on \eqref{eq:hankel_generalized} we have
\begin{equation}
\mathcal{H}_{k+1}(w_{[0,T-1]})=(I_{{k+1}}\otimes \mathcal{H}_1(w_{[0,T-1]}))\begin{bmatrix}
Q_0^\top \!&\! \cdots \!&\! Q_{k}^\top
\end{bmatrix}^\top.
\end{equation}
Since $\norm{I_{k+1}\otimes \mathcal{H}_1(w_{[0,T-1]})}=\norm{\mathcal{H}_1(w_{[0,T-1]})}\leq\sqrt{T}\varepsilon$ and  $\norm{\begin{bmatrix}
Q_0^\top & \cdots & Q_{k}^\top
\end{bmatrix}}=\sqrt{k+1}$, we have 
\begin{equation}
\norm{\mathcal{H}_{k+1}(w_{[0,T-1]})}\leq \sqrt{T(k+1)}\varepsilon.
\end{equation}
Using this inequality, we reduce \eqref{eq:th:ctrl_nec_noisy-3} to the following:
\begin{equation}
\label{eq:th:ctrl_nec_noisy-4}
\begin{split}
\sigma_\mathcal{D}\geq\tfrac{1}{\|d_A^{(k)}\|_1}\sigma_{n+m}(M_{A,B}^{(k)}D_A^{(k)}\mathcal{H}_{k+1}(u_{[0,T-1]})) \\
-\|M_{A}^{(k)}\|\sqrt{T(k+1)}\varepsilon.
\end{split}
\end{equation}
Now, since \eqref{eq:th:open_design_noisy_main} holds, \eqref{eq:th:ctrl_nec_noisy-4} implies that $\sigma_\mathcal{D}\geq \frac{\varepsilon\sqrt{T}}{\rho}$. Therefore, it follows from Lemma~\ref{lem:noisy_rad_lower(a)} that $\rad\Sigma_\mathcal{D}\leq \rho$.  \hfill \QED 

\subsection{Proof of Theorem~\ref{th:open_design_noisy}}

Suppose that $u_{[0,T-1]}$ satisfies \eqref{eq:th:open_design_noisy_main} for all members of $\Sigma_\text{pk}$. Let $(A,B)\in\Sigma_\text{pk}$ and $\mathcal{D}\in\mathfrak{B}(A,B,u_{[0,T-1]})$. It follows from Lemma~\ref{lem:open_design_noisy} that $\rad\Sigma_\mathcal{D}\leq\rho$. Since \mbox{$\rad(\Sigma_\mathcal{D}\cap\Sigma_\text{pk})\leq \rad\Sigma_\mathcal{D}$}, we have that $\mathcal{D}$ is $\Sigma_\text{pk}$--informative for $\rho$--accuracy identification. This argument holds for all $(A,B)\in\Sigma_\text{pk}$. Therefore, $u_{[0,T-1]}$ is $\Sigma_\text{pk}$--universal for $\rho$--accuracy identification. \hfill \QED

\subsection{Proof of Theorem~\ref{th:open_enabling_noisy}}
\label{app:I}

For this proof, we need some additional singular value inequalities, summarized in the following lemma.

\begin{lemma}[{\cite[Lem. 4.4 \& 4.5]{grcar2010matrix}}]
\label{lem:grcar}
Let $M\in\mathbb{R}^{p\times q}$ and $N\in\mathbb{R}^{q\times r}$ such that $MN\neq 0$. The following statements hold:
\begin{enumerate}[label=(\alph*),ref=\ref{lem:grcar}(\alph*)]
    \item\label{lem:grcar(a)} If either $M$ has full column rank or $N$ has full row rank, then $\sigma_*(M)\sigma_*(N)\leq \sigma_*(MN)$.
    \item\label{lem:grcar(c)} If $N$ has full row rank, then \mbox{$\sigma_*(N)\sigma_1(M)\leq\sigma_1(MN)$}.
\end{enumerate}
\end{lemma}

\textit{Proof of Theorem~\ref{th:open_enabling_noisy}:} (a) For the ``if'' part, suppose that $\gamma_\text{pk}>0$. Define $d_\text{pk}^{(k)}$ and $\mu_\text{pk}^{(k)}$ as follows:
\begin{equation}
d_\text{pk}^{(k)}\coloneqq \sup_{(A,B)\in\Sigma_\text{pk}}\tfrac{\|d_A^{(k)}\|_1}{\sigma_{m(k+1)}(D^{(k)}_A)},\  \mu_\text{pk}^{(k)}\coloneqq \sup_{(A,B)\in\Sigma_\text{pk}}\|M_A^{(k)}\|.
\end{equation}
We recall that $D^{(k)}_A$ is invertible by definition for all \mbox{$A\in\mathbb{R}^{n\times n}$}. This, combined with the fact that $\Sigma_\text{pk}$ is bounded, can be used to show that both $d_\text{pk}^{(k)}$ and $\mu_\text{pk}^{(k)}$ are finite and positive. Now, take input $u_{[0,T]}$ to be persistently exciting of order $k+1$ with a sufficiently large magnitude such that
\begin{equation}
\label{eq:choose_input}
\sigma_{m(k+1)}(\mathcal{H}_{k+1}(u_{[0,T-1]}))\geq \varepsilon\sqrt{T}\tfrac{d_\text{pk}^{(k)}}{\gamma_\text{pk}}\left(\mu_\text{pk}^{(k)}\sqrt{k+1}+\tfrac{1}{\rho}\right).
\end{equation}
Let $(A,B)\in\Sigma_\text{pk}$. Note that $\mathcal{H}_{k+1}(u_{[0,T-1]})$ and $M^{(k)}_{A,B}$ have full row rank and $D_{A}^{(k)}$ is invertible. As such, $M^{(k)}_{A,B}D_{A}^{(k)}\mathcal{H}_{k+1}(u_{[0,T-1]})$ has full row rank, and thus, is nonzero. We apply Lemma~\ref{lem:grcar(a)} twice to this matrix to obtain
\begin{equation}
\begin{split}
&\sigma_{n+m}(M^{(k)}_{A,B}D_{A}^{(k)}\mathcal{H}_{k+1}(u_{[0,T-1]}))\\
&\geq\sigma_{n+m}(M^{(k)}_{A,B} D_{A}^{(k)})\sigma_{m(k+1)}(\mathcal{H}_{k+1}(u_{[0,T-1]})) \\
&\geq\sigma_{n+m}(M^{(k)}_{A,B} )\sigma_{m(k+1)}(D_{A}^{(k)})\sigma_{m(k+1)}(\mathcal{H}_{k+1}(u_{[0,T-1]})) \\
&\geq\gamma_\text{pk}\sigma_{m(k+1)}(D_{A}^{(k)})\sigma_{m(k+1)}(\mathcal{H}_{k+1}(u_{[0,T-1]})).
\end{split}
\end{equation}
This, together with \eqref{eq:choose_input}, implies that
\begin{equation}
\label{eq:LHS_lowerbound}
\begin{split}
&\sigma_{n+m}(M^{(k)}_{A,B}D_{A}^{(k)}\mathcal{H}_{k+1}(u_{[0,T-1]}))\\
&\geq\sigma_{m(k+1)}(D_{A}^{(k)})\varepsilon\sqrt{T}d_\text{pk}^{(k)}\left(\mu_\text{pk}^{(k)}\sqrt{k+1}+\tfrac{1}{\rho}\right)\\
&\geq\varepsilon\sqrt{T}\|d_A^{(k)}\|_1\left(\mu_\text{pk}^{(k)}\sqrt{k+1}+\tfrac{1}{\rho}\right),
\end{split}
\end{equation}
where the last inequality follows from the fact that \mbox{$d_\text{pk}^{(k)}\geq \|d_A^{(k)}\|_1/\sigma_{m(k+1)}(D_{A}^{(k)})$}. Now, since $\|M_{A}^{(k)}\|\leq\mu_\text{pk}^{(k)}$, we have
\begin{equation}
\label{eq:RHS_upperbound}
\begin{split}
&\varepsilon\sqrt{T}\|d_A^{(k)}\|_1\left(\|M_{A}^{(k)}\|\sqrt{k+1}+\tfrac{1}{\rho}\right)\\
&\leq\varepsilon\sqrt{T}\|d_A^{(k)}\|_1\left(\mu_\text{pk}^{(k)}\sqrt{k+1}+\tfrac{1}{\rho}\right).
\end{split}
\end{equation}
Therefore, \eqref{eq:LHS_lowerbound} and \eqref{eq:RHS_upperbound} imply \eqref{eq:th:open_design_noisy_main}. Since the choice of $(A,B)\in\Sigma_\text{pk}$ was arbitrary, \eqref{eq:th:open_design_noisy_main} holds for all $(A,B)\in\Sigma_\text{pk}$. 

To prove the ``only if'' part, suppose that $\gamma_\text{pk}=0$. Let $T\in\mathbb{N}$ and $u_{[0,T-1]}\in\mathbb{R}^{mT}$. We will show that there exists \mbox{$(A,B)\in\Sigma_\text{pk}$} such that \eqref{eq:th:open_design_noisy_main} does not hold. Define \mbox{$\Delta_\text{pk}^{(k)}\coloneqq \sup_{(A,B)\in\Sigma_\text{pk}}\|D_A^{(k)}\|$}. It follows from the boundedness of $\Sigma_\text{pk}$ that $\Delta_\text{pk}^{(k)}$ is finite and positive. Since \mbox{$\gamma_\text{pk}=0$}, for every $\delta>0$ there exists $(A,B)\in\Sigma_\text{pk}$ such that $\sigma_{n+m}(M_{A,B}^{(k)})<\delta$. Let $(A,B)\in\Sigma_\text{pk}$ be such that
\begin{equation}
\label{eq:small_cont}
\sigma_{n+m}(M_{A,B}^{(k)})<\tfrac{\varepsilon\sqrt{T}(\sqrt{k+1}+\frac{1}{\rho})}{\Delta_\text{pk}^{\!(k)}\|\mathcal{H}_{k+1}(u_{[0,T-1]})\|}.
\end{equation}
This, together with Lemma~\ref{lem:horn(d)}, implies that the left-hand side of \eqref{eq:th:open_design_noisy_main} satisfies
\begin{equation}
\begin{split}
&\sigma_{n+m}(M^{(k)}_{A,B}D_{\!A}^{(k)}\mathcal{H}_{k+1}(u_{[0,T-1]}))\\
&\leq\sigma_{n+m}(M^{(k)}_{A,B})\|D_{\!A}^{(k)}\|\|\mathcal{H}_{k+1}(u_{[0,T-1]})\| \\
\label{eq:lhs_upper}
&< \varepsilon\sqrt{T}(\sqrt{k+1}+\tfrac{1}{\rho}).
\end{split}
\end{equation}
Note that, by definition, the last entry of $d_A^{(k)}$ is \mbox{$d_k=1$}. Therefore, $\|d_A^{(k)}\|_1\geq 1$. Moreover, $M^{(k)}_{A,B}$ is anti-block diagonal with one of the blocks equal to the identity matrix. This implies that $\|M^{(k)}_{A,B}\|\geq 1$. Thus, we observe that the right-hand side of \eqref{eq:th:open_design_noisy_main} admits the following lower bound:
\begin{equation}
\label{eq:rhs_lower}
\begin{split}
\varepsilon\sqrt{T}\|d_A^{(k)}\|_1(\|M_{A}^{(k)}\|\sqrt{k+1}+\tfrac{1}{\rho})
\geq \varepsilon\sqrt{T}(\sqrt{k+1}+\tfrac{1}{\rho}).
\end{split}
\end{equation}
Now, this and \eqref{eq:lhs_upper} imply that \eqref{eq:th:open_design_noisy_main} does not hold. 

(b) Suppose that $\gamma_\text{pk}>0$. It follows from part (a) that there exists $u_{[0,T]}$ such that \eqref{eq:th:open_design_noisy_main} holds for all $(A,B)\in\Sigma_\text{pk}$. Now, in view of Theorem~\ref{th:open_design_noisy}, $u_{[0,T]}$ is $\Sigma_\text{pk}$--universal for $\rho$--accuracy identification. \hfill \QED

\subsection{Proof of Theorem~\ref{th:helping_in_exp}}

Let $\Sigma_\text{pk}=\Sigma_\text{pk}^\prime\cup\Sigma_\text{pk}^{\prime\prime}$. Suppose that $(A,B)\in\Sigma_\text{pk}^\prime$. Then, for every \mbox{$\mathcal{D}\in\mathfrak{B}(A,B,u_{[0,T-1]})$} we have $\rad\Sigma_\mathcal{D}\leq\rho$. This implies that $\rad(\Sigma_\mathcal{D}\cap\Sigma_\text{pk})\leq\rho$. Now, suppose that \mbox{$(A,B)\in\Sigma_\text{pk}^{\prime\prime}$}. Let $\mathcal{D}\in\mathfrak{B}(A,B,u_{[0,T-1]})$. Aiming for a contradiction, assume that $\rad(\Sigma_\mathcal{D}\cap\Sigma_\text{pk})>\rho$. Since $\rad\Sigma_\text{pk}^{\prime\prime}\leq \rho$, this implies that $\Sigma_\mathcal{D}\cap\Sigma_\text{pk}^\prime$ is nonempty. Therefore, there exists $(\hat{A},\hat{B})\in\Sigma_\text{pk}^\prime$ such that $\mathcal{D}\in\mathfrak{B}_T(\hat{A},\hat{B})$. This implies that $\rad \Sigma_\mathcal{D}\leq \rho$. This contradicts $\rad(\Sigma_\mathcal{D}\cap\Sigma_\text{pk})>\rho$. Therefore, we have $\rad(\Sigma_\mathcal{D}\cap\Sigma_\text{pk})\leq\rho$. \hfill \QED

\subsection{Proof of Proposition~\ref{prop:impossible_noisy}}

For $(A,B)\in\Sigma$ and $r>0$, let $\mathcal{B}_r(A,B)$ denote the set of all matrix pairs $(M,N)$ such that $\norm{\begin{bmatrix}
M-A & N-B
\end{bmatrix}}\leq r$. 

To prove the statement, we use contraposition. Suppose that $\Sigma_\text{pk}$ is not uniformly discrete. We will show that $\Sigma_\text{pk}$ does not enable universal experiment design for identification. Let $T\in\mathbb{N}$ and \mbox{$u_{[0,T-1]}\in\mathbb{R}^{mT}$}. Since $\Sigma_\text{pk}$ is not uniformly discrete, there exists $(A,B)\in\Sigma_\text{pk}$ such that for every \mbox{$r>0$} we have that $\mathcal{B}_r(A,B)\cap\Sigma_\text{pk}$ is not a singleton. Let \mbox{$\mathcal{D}=(u_{[0,T-1]},x_{[0,T]})\in\mathfrak{B}(A,B,u_{[0,T-1]})$} and the noise signal $w_{[0,T-1]}=0$ satisfy \eqref{eq:1} for all $t\in[0,T-1]$. Recall the definition of $N_\text{in}$ from \eqref{eq:N_in} and consider its partitioned form $N_\text{in}=\begin{bmatrix}
N_{11} & N_{12} \\
N_{21} & N_{22}
\end{bmatrix}$, where $N_{11}\in\mathbb{S}^n$, \mbox{$N_{22}\in\mathbb{S}^{n+m}$}, and \mbox{$N_{12}=N_{21}^\top\in\mathbb{R}^{n\times (n+m)}$}. One can take the same steps as in the proof of Lemma~\ref{lem:noisy_rad_lower} to show that $N|N_{22}=\varepsilon^2 I$. Now, it follows from \cite[Thm. 3.2(c)]{HenkQMI2023} that $\mathcal{Z}_{n+m}(N_\text{in})$ has a nonempty interior. We note that \mbox{$\begin{bmatrix}
A & B
\end{bmatrix}^\top\in\mathcal{Z}_{n+m}(N_\text{in})$}. Thus, in view of Lemma~\ref{lem:exp_qmi}, there exists $r>0$ such that $\mathcal{B}_r(A,B)\subseteq\Sigma_\mathcal{D}$. Since $\mathcal{B}_r(A,B)\cap\Sigma_\text{pk}$ is not a singleton, $\Sigma_\mathcal{D}\cap\Sigma_\text{pk}$ is also not a singleton. Therefore, $\mathcal{D}$ is not $\Sigma_\text{pk}$--informative for identification. Since this argument holds for all $T\in\mathbb{N}$ and $u_{[0,T-1]}\in\mathbb{R}^{mT}$, the set $\Sigma_\text{pk}$ does not enable universal experiment design for identification. \hfill \QED

\subsection{Proof of Proposition~\ref{prop:data_best_noisy}}

Suppose that \eqref{eq:rank_data_generalized_noisy} holds for all nonzero \mbox{$(F,G)\in\Sigma_\text{pk}-\Sigma_\text{pk}$}. We show that $\Sigma_\mathcal{D}\cap\Sigma_\text{pk}$ is a singleton. For this, let $(A_1,B_1),(A_2,B_2)\in\Sigma_\text{pk}\cap\Sigma_\mathcal{D}$. Let \mbox{$w^i(0),\ldots,w^i(T-1)\in\mathbb{R}^{n}$}, for $i=1,2$, be two noise sequences such that $\|w^i(t)\|\leq\varepsilon, \forall t\in[0,T-1]$, and
\begin{equation}
\label{eq:rank_data_generalized_noisy-1}
\mathcal{H}_1(\mathcal{H}_1(x_{[1,T]}))=\begin{bmatrix}
A_i & B_i
\end{bmatrix}\begin{bmatrix}
\mathcal{H}_1(x_{[0,T-1]}) \\
\mathcal{H}_1(u_{[0,T-1]})
\end{bmatrix}+\mathcal{H}_1(w^i_{[0,T-1]}) 
\end{equation}
for $i=1,2$. Let $\bar{F}=A_1-A_2$ and $\bar{G}=B_1-B_2$, and note that we have $(\bar{F},\bar{G})\in\Sigma_\text{pk}-\Sigma_\text{pk}$. Due to \eqref{eq:rank_data_generalized_noisy-1} we have
\begin{equation}
\begin{bmatrix}
\bar{F} & \bar{G}
\end{bmatrix}\begin{bmatrix}
\mathcal{H}_1(x_{[0,T-1]}) \\
\mathcal{H}_1(u_{[0,T-1]})
\end{bmatrix}=\mathcal{H}_1(w^2_{[0,T-1]})-\mathcal{H}_1(w^1_{[0,T-1]}).
\end{equation}
We observe that $\|\mathcal{H}_1(w^i_{[0,T-1]})\|^2\leq\sum_{t=0}^{T-1} \|w^i(t)\|^2\leq T\varepsilon^2$
for $i=1,2$. Therefore, we have
\begin{equation}
\norm{\begin{bmatrix}
\bar{F} & \bar{G}
\end{bmatrix}\begin{bmatrix}
\mathcal{H}_1(x_{[0,T-1]}) \\
\mathcal{H}_1(u_{[0,T-1]})
\end{bmatrix}}\leq 2\sqrt{T}\varepsilon.
\end{equation}
Since \eqref{eq:rank_data_generalized_noisy} holds for all nonzero $(F,G)\in\Sigma_\text{pk}-\Sigma_\text{pk}$, we have
 $\norm{\begin{bmatrix}
\bar{F} & \bar{G}
\end{bmatrix}}=0$. This implies that $A_1=A_2$ and $B_1=B_2$. Since this argument holds for any two elements of $\Sigma_\mathcal{D}\cap\Sigma_\text{pk}$, we have that $\Sigma_\mathcal{D}\cap\Sigma_\text{pk}$ is a singleton. \hfill \QED

\subsection{Proof of Theorem~\ref{th:GM_noisy_generalized}}

Suppose that $u_{[0,T-1]}$ satisfies~\eqref{eq:th:GM_noisy_generalized-1} for all $(A,B)\in\Sigma_\text{pk}$ and all nonzero $(F,G)\in\Sigma_\text{pk}-\Sigma_\text{pk}$. We show that $u_{[0,T-1]}$ is $\Sigma_\textup{pk}$--universal for identification. For this, let $(A,B)\in\Sigma_\text{pk}$ and \mbox{$\mathcal{D}=(u_{[0,T-1]},x_{[0,T]})\in\mathfrak{B}(A,B,u_{[0,T-1]})$}. Let $w_{[0,T-1]}$ be the noise signal satisfying 
\begin{equation}
\mathcal{H}_1(x_{[1,T]})\!=\!A\mathcal{H}_1(x_{[0,T-1]})+B\mathcal{H}_1(u_{[0,T-1]})+\mathcal{H}_1(w_{[0,T-1]}).
\end{equation}
Let $(F,G)\in\Sigma_\text{pk}-\Sigma_\text{pk}$ be nonzero. According to~\eqref{eq:GM_formula} and the fact that $\|\sum_{i=0}^{k}d_iQ_i\|\leq\|d_A^{(k)}\|_1$, we have
\begin{equation}
\label{eq:th:GM_noisy_generalized-2}
\begin{split}
&\norm{\begin{bmatrix}
F & G
\end{bmatrix}\begin{bmatrix}
\mathcal{H}_{1}(x_{[0,T-1]}) \\ \mathcal{H}_{1}(u_{[0,T-1]})
\end{bmatrix}}\|d_A\|_1\\
&\geq\norm{\begin{bmatrix}
F & G
\end{bmatrix}M_{A,B}D_A\mathcal{H}_{n+1}(u_{[0,T-1]})
}\\
&-\norm{\begin{bmatrix}
F & 0
\end{bmatrix}M_{A,I}D_A\mathcal{H}_{n+1}(w_{[0,T-1]})}.
\end{split}
\end{equation}
We recall that the noise model \eqref{eq:ass1} implies $\norm{\mathcal{H}_{n+1}(w_{[0,T-1]})}\leq \sqrt{T(n+1)}\varepsilon$. Using this, together with $\norm{D_A}\leq\norm{d_A}_1$, we reduce \eqref{eq:th:GM_noisy_generalized-2} to the following:
\begin{equation}
\label{eq:th:GM_noisy_generalized-3}
\begin{split}
&\norm{\begin{bmatrix}
F & G
\end{bmatrix}\begin{bmatrix}
\mathcal{H}_{1}(x_{[0,T-1]}) \\ \mathcal{H}_{1}(u_{[0,T-1]})
\end{bmatrix}}\\
&\geq\tfrac{1}{\norm{d_A}_1}\norm{\begin{bmatrix}
F & G
\end{bmatrix}M_{A,B}D_A\mathcal{H}_{n+1}(u_{[0,T-1]})
}\\
&-\varepsilon\sqrt{T(n+1)}\norm{F}\norm{M_{A,I}}.
\end{split}
\end{equation}
Now, since \eqref{eq:th:GM_noisy_generalized-1} holds, \eqref{eq:th:GM_noisy_generalized-3} implies that
\begin{equation}
\norm{\begin{bmatrix}
F & G
\end{bmatrix}\begin{bmatrix}
\mathcal{H}_1(x_{[0,T-1]}) \\
\mathcal{H}_1(u_{[0,T-1]})
\end{bmatrix}}> 2\sqrt{T}\varepsilon.
\end{equation}
Therefore, since this holds for all nonzero $(F,G)\in\Sigma_\text{pk}-\Sigma_\text{pk}$, it follows from Proposition~\ref{prop:data_best_noisy} that $\mathcal{D}$ is \mbox{$\Sigma_\text{pk}$--informative} for identification. Since this argument holds for all $(A,B)\in\Sigma_\text{pk}$, input $u_{[0,T-1]}$ is $\Sigma_\textup{pk}$--universal for identification.\hfill \QED